\documentclass[12pt]{article}
\usepackage{amsmath,amssymb,amsthm,amsfonts,enumerate,graphicx,url}
\usepackage{hyperref}
\usepackage[matrix,arrow,curve]{xy}
\hypersetup{
   colorlinks   = true, 
   urlcolor     = blue, 
   linkcolor    = blue, 
   citecolor   = red 
}
\usepackage[usenames,dvipsnames]{color}
\input{xy}\xyoption{all}
\newcommand{\jonly}[1]{}
\newcommand{\aronly}[1]{#1}
\newcommand{\talkonly}[1]{}
\newcommand{\talkno}[1]{#1}

\def\Z{{\mathbb Z}} \def\R{{\mathbb R}}  
\long\def\comment#1\endcomment{}
\def\Hom{\mathop{\fam0 Hom}}

\def\Img{\mathop{\fam0 Im}}
\def\Ker{\mathop{\fam0 Ker}}

\renewcommand{\t}[1]{\ensuremath{\widetilde{#1}}}

\theoremstyle{plain}
\newtheorem{theorem}{Theorem}[section]
\newtheorem{lemma}[theorem]{Lemma}

\newtheorem{statement}[theorem]{Statement}
\newtheorem{conjecture}[theorem]{Conjecture}

\theoremstyle{definition}
\newtheorem{remark}[theorem]{Remark}
 
\begin{document}
 
\talkonly{
\huge
\centerline{\bf Embeddability of joinpowers,} 
\smallskip
\centerline{\bf and minimal rank of partial matrices}
\medskip
\centerline{\it Arkadiy Skopenkov and Oleg Styrt}
\bigskip
}

\title{Embeddability of joinpowers \\ and minimal rank of partial matrices\footnote{We are grateful to E. Alkin, B.H. An, S. Dzhenzher, R. Fulek, T. Garaev, R. Karasev, A. Miroshnikov, and anonymous referees for useful discussions.}}


\author{Arkadiy Skopenkov and Oleg Styrt\footnote{Both authors: 
Moscow Center for Continuous Mathematical Education; supported by Russian Science Foundation,  Grant N 25-21-00685, \url{https://rscf.ru/en/project/25-21-00685/}.
A. Skopenkov: \texttt{skopenko@mccme.ru}, \texttt{https://users.mccme.ru/skopenko}.}}

\date{}
\maketitle

\begin{abstract} 
A general position map $f:K\to M$ of a $k$-dimensional simplicial complex to a $2k$-dimensional manifold (for $k=1$, of a graph to a surface) is a \textit{$\Z_2$-embedding} if $|f\sigma \cap f\tau|$ is even for any vertex-disjoint 
$k$-dimensional faces (for $k=1$, edges) $\sigma,\tau$. 
We present criteria for $\Z_2$-embeddability of certain $k$-dimensional complex (for $k=1$, of any graph) to a $2k$-dimensional manifold. 
These criteria are 

$\bullet$ a `Kuratowski-type' version of the Fulek-Kyn\v{c}l-Bikeev criteria (for $k=1$), and 

$\bullet$ a converse to the Dzhenzher-Skopenkov necessary condition (for $k>1$). 
 
Our higher-dimensional criterion allows us to reduce the modulo 2 K\"uhnel problem on embeddings to a purely algebraic problem. 

Our proof is interplay between geometric topology, combinatorics and linear algebra. 
It is based on calculation of generators in the homology of certain configuration space (the deleted product) of certain complex (joinpower). 
\end{abstract}



\tableofcontents

\section{Introduction and main results}\label{s:mr}

\subsection*{Description of main results}

We abbreviate `$k$-dimensional' to just `$k$-'. 

In this paper a {\bf complex} is the set of some closed faces of some simplex (a longer term is finite simplicial complex). 
We identify this set with the union of those faces (i.e. with the \emph{body} of the complex). 
A {\bf $k$-complex} is the union of some closed faces of dimension at most $k$ (for $k=1$ this is a graph). 

A general position map $K\to M$ of a $k$-complex to a $2k$-manifold (see the definition e.g. in \cite[\S1.1]{Sk24}) is called a \textbf{$\Z_2$-embedding} if $|f\sigma \cap f\tau|$ is even for any vertex-disjoint $k$-faces $\sigma,\tau$. 

We present criteria for $\Z_2$-embeddability of certain $k$-complex (for $k=1$, of any graph) to $2k$-manifolds. 

\talkonly{\newpage}
These criteria are 

$\bullet$ a `Kuratowski-type' version of the Fulek-Kyn\v cl-Bikeev criteria, i.e., a version involving only obstructions coming from $K_5$ and $K_{3,3}$ (for $k=1$\talkno{; Theorems \ref{t:embntfs}, \ref{t:embntf} and \ref{t:embntfg}}), and  

$\bullet$ a converse to the Dzhenzher-Skopenkov necessary condition (for $k>1$\talkno{; Theorem \ref{t:embntfk}}). 

Our higher-dimensional criterion allows us to reduce the modulo 2 K\"uhnel problem on embeddings to a purely algebraic problem\talkno{ (Conjecture \ref{p:open})}. 

Our proof is an interplay between geometric topology, combinatorics, and linear algebra. 
It is based on the calculation of generators in the homology of a certain configuration space (the deleted product) of a certain complex (joinpower), see Theorems \ref{r:dedu}, \ref{l:h2sym},   
and Remark \ref{r:othcom}. 

\subsection*{Some motivation}

\talkno{This subsection is formally not used below.} 
 
The classical Heawood inequality provides a lower estimation for a number $g$ such that the complete graph $K_n$ on $n$ vertices embeds into the sphere with $g$ handles. 
A higher-dimensional analogue of the Heawood inequality is the K\"uhnel conjecture on embeddings.
In a simplified form it asks for an analogous estimation of the `complexity' (i.e., of $\dim H_k(M;\Z_2)$) of a $2k$-manifold $M$ into which the union of all $k$-faces of $n$-simplex is embedded. 
See the detailed discussion of the conjecture in \cite[Remark 1.3]{DS22}, cf. \cite[\S4]{Ku23}.  

Recent results on the conjecture work not only for embeddings but for $\Z_2$-embeddings (\cite[Theorem 1]{PT19}, \cite{DS22}; see \cite{KS21} for a simplified direct proof of \cite[Theorem 1]{PT19}).  
These results involve a necessary condition for $\Z_2$-embeddability of certain $k$-complexes to $2k$-manifolds (the `only if' parts of Theorems \ref{t:embntfs}, \ref{t:embntf} and \ref{t:embntfk}; similarly one proves the `only if' part of Theorem \ref{t:embntfg}, which part was not stated). 
This condition is stated (implicitly in \cite{PT19} and explicitly in \cite{DS22, KS21}) in terms of finding the minimal rank of a matrix with certain relations on its entries, and having a certain special form. 
So this condition is related to the low rank matrix completion problem, and to the Netflix problem from machine learning. 
\emph{Our main result} is the sufficiency of this condition\talkno{ (the `if' parts of Theorems \ref{t:embntfs}--\ref{t:embntfg} and \ref{t:embntfk})}.  
\aronly{See survey \cite{DGN+} and the references therein.} 


\talkonly{\newpage}
Known results on embeddability of $k$-complexes in $2k$-manifolds are described in \cite[Remark 1.1.7.bc]{Sk24}. 
Some motivations for studies of $\Z_2$-embeddability are presented in \cite[\S1.2]{Sk24}. 
The particular case of $\Z_2$-embeddability of graphs to surfaces is actively studied\talkno{, see survey \cite{Sc13} and recent papers \cite{Ky16, FK19,  Bi21}}. 

\emph{Our main instrument} is a description of generators in the homology of certain configuration space (deleted product), see Theorem \ref{l:h2sym} (and Theorem \ref{r:dedu} due to \cite{Ho07}). 
For related results 

$\bullet$ on graphs see \cite{Ni00, ST03, Ho07, BF09, Ho09, FH10}, \cite[Theorem 3.16]{KP11}, \cite{Ky16, HNT, AK20, Dz25}, and references in \cite{AK20};  

$\bullet$ on higher-dimensional complexes see \cite[\S19]{Di08}, \cite[Theorem 3.8]{Me06}. 

(References to papers published earlier than 2000 can be found in the above papers.)




\subsection*{Notation and conventions}

Below all matrices have $\Z_2$-entries.
Assume that $n\ge3$. 
Let  $[n]:=\{1,2,\ldots,n\}$.   

Denote by $K_{n,n}$ the complete bipartite graph with parts $[n]$ and a copy $[n]':=\{k'\ :\ k\in[n]\}$ of $[n]$.  

Denote by $I_n$ the identity matrix of size $n$. 
Denote by $H_n$ the square matrix of size $2n$ whose off-diagonal $2\times2$-blocks are zero matrices, and whose diagonal $2\times2$-blocks are
$\begin{pmatrix} 0 & 1 \\ 1 & 0 \end{pmatrix}$.  
For a matrix $A$ let $A_{P,Q}$ be the entry in the row of $P$ and column of $Q$.

A (simplicial) \textbf{$1$-cycle} (modulo 2) in a graph $K$ is a set $C$ of edges such that every vertex is contained in an even number of edges from $C$. 
A (simplicial) \textbf{$k$-cycle} (modulo 2) in a complex 
$K$ is a set $C$ of $k$-faces such that every $(k-1)$-face is contained in an even number of faces from $C$. 

We omit $\Z_2$-coefficients of homology groups considered below.
Denote by $H_k(K)$ the group of all $k$-cycles in a $k$-complex $K$, with the sum modulo 2 operation. 
For a $2k$-manifold $M$ denote by $H_k(M)$ the modulo 2 \emph{homology group}, and by $\cap_M:H_k(M)\times H_k(M)\to\Z_2$ the modulo 2 \emph{intersection form} (for definitions accessible to non-specialists see 
\cite[\S6, \S10]{Sk20}).

\subsection*{Results for graphs}

\talkno{Our results for graphs are Theorems \ref{t:embntfs}, \ref{t:embntf} and \ref{t:embntfg} (of which Theorem \ref{t:embntfs} is a particular case of Theorem \ref{t:embntfg}, and Theorem \ref{t:embntf} is a particular case of both Theorems \ref{t:embntfg} and \ref{t:embntfk}).} 

Let $M_m$ be the sphere with $m$ M\"obius films, 
$S_g$ the sphere with $g$ handles, and $N$ either $M_m$ or $S_g$.   
Denote 
$$\Omega_{M_m}:=I_m,\quad \beta(M_m):=m,\quad  \Omega_{S_g}:=H_g,\quad\text{and}\quad\beta(S_g):=2g.$$  

Let $A$ be a symmetric square matrix of $\binom{n}{3}$ rows (whose rows and whose columns correspond to all 3-element subsets of $[n]$).  
The matrix $A$ is \talkno{said to be} 

$\bullet$ \textbf{independent} if $A_{P,Q} = 0$ whenever $P \cap Q = \varnothing$;  

$\bullet$ \textbf{additive} if 
$$A_{F-a,Q}+A_{F-b,Q}+A_{F-c,Q}+A_{F-d,Q}=0$$ 
for any 4-element subset $F=\{a,b,c,d\}\subset[n]$ and 3-element subset $Q\subset[n]$ (here we abbreviate $\{a\}$ to $a$). 

$\bullet$ \textbf{non-trivial} if for any $5$-element subset $F\subset[n]$ and\footnote{\label{f:indep} Here and below after Theorem \ref{t:embntfs} 
\newline
$\bullet$ `for any vertex $v$ / edge $e$' is equivalent to `for some vertex $v$ / edge $e$' \cite[Lemma 20]{PT19}, \cite[Proposition 2.2 and Remark 2.3]{KS21}, \cite[Remark 4.4]{DS22};  
\newline
$\bullet$ strictly speaking, the properties of matrix $A$ should be called `$K_5$-independence',  `$K_{3,3}$-independence', etc (for the sake of readability, we shorten the notation).} 
$v\in F$ we have $S_{F,v}A=1$, where $S_{F,v}A$ is the sum of $A_{P,Q}$ over all the $3$ unordered pairs $\{P,Q\}$ of 
$3$-element subsets $P,Q\subset F$ such that $P\cap Q=\{v\}$. 
 
\begin{theorem}\label{t:embntfs}  
A $\Z_2$-embedding $K_n\to N$ exists if and only if there is a $\beta(N)\times{n\choose3}$ matrix $Y$ such that $Y^T\Omega_NY$ is an independent additive non-trivial matrix. 
\end{theorem}


\talkonly{\newpage
{\bf The result for complete bipartite graph}}

Let $A$ be a symmetric square matrix of $\binom{n}{2}^2$ rows (whose rows and whose columns correspond to all cycles of length 4 in $K_{n,n}$).
The matrix $A$ is said to be 

$\bullet$ \textbf{independent} if $A_{P,Q} = 0$ whenever $P$ and $Q$ are vertex-disjoint; 

$\bullet$ \textbf{additive} if $A_{P,Q} = A_{X,Q} + A_{Y,Q}$  whenever $P = X \oplus Y$ (as sets $P,X,Y$ of edges).  
 
$\bullet$ \textbf{non-trivial} if for any subgraph $X$ of $K_{n,n}$ isomorphic to $K_{3,3}$ and  edge $e$ of $X$ we have $S_{X,e}A=1$, where $S_{X,e}A$ is the sum of $A_{P,Q}$ over all the 2 unordered pairs $\{P,Q\}$ of 
cycles $P,Q\subset X$ of length 4 whose only common edge is $e$. 
   
\begin{theorem}\label{t:embntf}  
A $\Z_2$-embedding $K_{n,n}\to N$ exists if and only if there is a $\beta(N)\times{n\choose2}^2$ matrix $Y$ such that $Y^T\Omega_NY$ is an independent additive non-trivial matrix. 
\end{theorem}


\talkno{

\begin{theorem}\label{t:embntfg} A $\Z_2$-embedding $K\to N$ of a graph $K$ exists if and only if there is a homomorphism $y:H_1(K)\to H_1(N)$ such that 
 
(independence) $y(P)\cap_N y(Q) = 0$ whenever $P$ and $Q$ are vertex-disjoint; 
 
($K_5$-non-triviality) for any subgraph $F$ homeomorphic to $K_5$, and vertex $v$ of $F$ we have $S_{F,v}y=1$, where $S_{F,v}y$ is the sum of $y(P)\cap_N y(Q)$ over all the 3 unordered pairs $\{P,Q\}$ of cycles $P,Q\subset F$ of length\footnote{Here and in ($K_{3,3}$-non-triviality) the cycles of length 3 and 4 are not the cycles of that length in $K$, but rather they are cycles with 3 and 4, respectively, vertices of degree more than 2 in $F$ and in $X$.} 
3 such that $v$ is the only common vertex of $P$ and $Q$; and
 
($K_{3,3}$-non-triviality) for any subgraph $X$ homeomorphic to $K_{3,3}$, and edge $e$ of $X$ we have $S_{X,e}y=1$, where $S_{X,e}y$ is the sum of $y(P)\cap_N y(Q)$ over all the 2 unordered pairs $\{P,Q\}$ of cycles $P,Q\subset X$ of length 4 such that $e$ is the only common edge of $P$ and $Q$. 
\end{theorem}

\begin{remark}\label{r:matr}  
(a) (on the novelty and idea of proof) The `only if' part of Theorem \ref{t:embntfg} is essentially known, and is analogous to \cite[Remark 1.6]{DS22}, which in turn is analogous to \cite{FK19, Bi21}. 
Proof of the `if' part is analogous to the proof of the `if' part of Theorem \ref{t:embntfk} for $k=1$ in \S\ref{s:meta}, except that reference to Theorem \ref{l:h2sym}.b is replaced by reference to Theorem \ref{r:dedu}. 
See more comments on the proof in Remark \ref{r:othcom}.c. 

(b) \jonly{For a matrix form of Theorem \ref{t:embntfg} (and so of Theorems \ref{t:embntfs},  \ref{t:embntf}) see \cite[Remark 1.4.b]{SS23}.}
\aronly{In matrix form, the condition of Theorem \ref{t:embntfg} means that there are a set of $r$ generators of $H_1(K)$, and a $\beta(N)\times r$ matrix $Y$ such that $A:=Y^T\Omega_NY$ is independent additive non-trivial (as defined below). 

Let $A$ be a symmetric $r\times r$ matrix whose rows and whose columns correspond to some generators of $H_1(K)$. 
The matrix $A$ is said to be 

$\bullet$ \emph{independent}, if $A_{P,Q} = 0$ whenever $P$ and $Q$ are vertex-disjoint;     

$\bullet$ \emph{additive}, if $A_{P,Q} = A_{P_1,Q} + \ldots + A_{P_t,Q}$  whenever $P = P_1 \oplus \ldots\oplus P_t$ (as sets of edges). 

If $X,Y\in H_1(K)$, then express them as a sum of the generators, $X=X_1+\ldots+X_p$, $Y=Y_1+\ldots+Y_q$, and define $A_{X,Y}:=\sum_{i,j}A_{X_i,Y_j}$. 
The matrix $A$ is said to be \emph{$K_5$-non-trivial} (\emph{$K_{3,3}$-non-trivial}), if the $K_5$-non-triviality ($K_{3,3}$-non-triviality) of Theorem \ref{t:embntfg} holds replacing $y(P)\cap_N y(Q)$ by $A_{P,Q}$.} 
\end{remark}

}

\talkonly{\newpage}
\subsection*{The main result for higher dimensions}

We consider piecewise-linear (PL) manifolds as defined in \cite{RS72}. 
Let $M$ be a closed $2k$-manifold. 
Denote $\beta(M):=\dim H_k(M)$. 
Denote $\Omega_M:=I_{\beta(M)}$ if there is $x\in H_k(M)$ such that $x\cap_Mx=1$, and  $\Omega_M:=H_{\beta(M)/2}$ otherwise, when $\beta(M)$ is known to be even ($\Omega_M$ is the matrix of $\cap_M$ in some basis).  

Denote $D^j:=[0,1]^j$ and $S^{j-1}:=\partial D^j$. 
A complex $M$ is called \emph{$m$-connected} if for any $j=0,1,\ldots,m$ every continuous map $f:S^j\to M$ extends  continuously over $D^{j+1}$.

Let $[n]^{*k+1}$ be the $k$-complex 

$\bullet$ whose vertex set is the $(k+1)\times n$ grid $\{0,\ldots,k\}\times[n]$, 

$\bullet$ in which every $k+1$ vertices from different rows (i.e. with different first coordinates) span a $k$-simplex. 

For $k=1$ this is $K_{n,n}$. 
For geometric interpretation see \cite[Proposition~4.2.4]{Ma03}.
A \textbf{$k$-octahedron} is the set of $k$-faces of a subcomplex (of $[n]^{*k+1}$) isomorphic to $[2]^{*k+1} \cong S^k$.
For $2$-element subsets $P_0,\ldots,P_k\subset[n]$ such a subcomplex is defined by the set $0 \times P_0 \sqcup \ldots \sqcup k \times P_k$ of its vertices. 
 
\talkonly{\newpage}
  
Let $A$ be a symmetric square matrix of $\binom{n}{2}^{k+1}$ rows (whose rows and whose columns correspond to all $k$-octahedra).
The matrix $A$ is said to be 

$\bullet$ \textbf{independent} if $A_{P,Q}=0$ whenever $P$ and $Q$ are vertex-disjoint; 
 
$\bullet$ \textbf{additive} if $A_{P,Q} = A_{X,Q} + A_{Y,Q}$  whenever $P = X \oplus Y$  (as sets $P,X,Y$ of $k$-faces);   

$\bullet$ \textbf{non-trivial} if for any subcomplex $X$ of $[n]^{*k+1}$ isomorphic to $[3]^{*k+1}$ and $k$-face $e$ of $X$ we have $S_{X,e}A=1$, where by $S_{X,e}A$ we denote the sum of $A_{P,Q}$ over all $2^k$ unordered pairs $\{P,Q\}$ of $k$-octahedra $P,Q\subset X$ whose only common $k$-face is $e$. 
  
\begin{theorem}\label{t:embntfk} Let $M$ be a closed $(k-1)$-connected $2k$-manifold.
A $\Z_2$-embedding $[n]^{*k+1}\to M$ exists if and only if there is a $\beta(M)\times{n\choose2}^{k+1}$ matrix $Y$ such that $Y^T\Omega_MY$ is an independent additive non-trivial matrix.
\end{theorem}

\talkonly{\newpage}
\begin{remark}\label{r:othcom} 
(a) (on the novelty and idea of proof) 
The `only if' part of Theorem \ref{t:embntfk} is \cite[Remark 1.9b]{DS22}. 
Our main achievement is the `if' part (\S\ref{s:meta}). 
For this we use a homological $\Z_2$-embeddability criterion (Theorem \ref{t:homembcri} by E. Kogan; for completeness, we present in \S\ref{s:meta} a standard deduction from known results). 
Our main idea is that the independence and non-triviality of the matrix $Y^T\Omega_MY$ correspond to generators of the homology group of a certain configuration space (Theorem \ref{l:h2sym}).  

(b) A version of Theorem \ref{t:embntfk} for the $k$-skeleton of $n$-simplex instead of $[n]^{*k+1}$ is stated analogously (cf. Theorem \ref{t:embntfs}). 
Such a version is presumably false for $k>1$ because the relevant homology group of  configuration space is complicated.
So Theorem \ref{t:embntfk} is not as trivial a corollary of Theorem \ref{t:homembcri} as it may seem. 
 
(c) Theorem \ref{t:embntfk} generalizes to joins of finite sets of different sizes $n_0,\ldots,n_k\ge3$:   
\emph{a $\Z_2$-embedding $[n_0]*\ldots*[n_k]\to M$ exists if and only if there is a 
$\beta(M)\times{n_0\choose2}\ldots{n_k\choose2}$ matrix $Y$ such that $Y^T\Omega_MY$ is an independent additive non-trivial matrix}.
The proof of Theorem \ref{t:embntfk} generalizes trivially to a proof of the generalization, 
see Remarks \ref{dif2}, \ref{dif3}, \ref{dif4}.
In those Remarks, the changes of the proofs that repeat changes of the statements are omitted.
\end{remark}

\talkonly{\newpage}

\begin{conjecture}\label{p:open} (a) (asymptotic mod 2 K\"uhnel conjecture)
\emph{For any integer $k>0$ there is $c_k>0$ such that if $[n]^{*k+1}$ has a $\Z_2$-embedding to a closed $2k$-manifold $M$, then $\beta(M)\ge c_k n^{k+1}$.}

By Theorem \ref{t:embntfk}, this is equivalent to the following algebraic conjecture: 

\emph{For any integer $k>0$ there is $c_k>0$ such that if there is a $d\times{n\choose2}^{k+1}$ matrix $Y$ such that either $Y^TY$ or $Y^TH_{d/2}Y$ is independent additive non-trivial, then $d\ge c_k n^{k+1}$.}

Both conjectures are proved for $k=1$ \cite{FK19}, and for any $k$ replacing $n^{k+1}$ by $n^2$ \cite{DS22}.  
 
\talkonly{\medskip}

(b) (analogue of Theorem \ref{t:embntfk} for embeddings) 
\emph{Assume that $k\ge3$ and $M$ is a closed $(k-1)$-connected $2k$-manifold.
An embedding $[n]^{*k+1}\to M$ exists if and only if there is a $\beta(M)\times{n\choose2}^{k+1}$ matrix $Y$ with integer entries such that $Y^T\Omega_{M,\Z}Y$ is an independent additive non-trivial matrix.}

Here $\Omega_{M,\Z}$ is the matrix of the intersection form of $M$ in some basis of $H_k(M;\Z)\cong\Z^{\beta(M)}$\talkno{ (the isomorphism holds since $M$ is $(k-1)$-connected)}. 

\talkno{
Let $A$ be a $(-1)^k$-symmetric square matrix with integer entries, of $\binom{n}{2}^{k+1}$ rows. 
Independence is defined analogously to the mod 2 case.  
The matrix $A$ is said to be 

$\bullet$ \emph{additive} if $A_{P,Q} = A_{X,Q} + A_{Y,Q}$  whenever $P = X + Y$  (as integer $k$-cycles);   

$\bullet$ \emph{non-trivial} if for any subcomplex $X$ of $[n]^{*k+1}$ isomorphic to $[3]^{*k+1}$ and $k$-face $e$ of $X$ the sum $S_{X,e}A$ is odd. 

The `only if' part of this conjecture is proved analogously to \cite[Theorem 1.8 and Remark 1.9b]{DS22}. 
}
\end{conjecture}

\talkonly{\newpage}
\section{Generators in homology of the complex}\label{s:genpow}

Of this section, only Lemmas \ref{l:genegn}, \ref{l:tkn4} and \ref{l:bas} 
are used later. 

\talkonly{A \textbf{$k$-cycle} in a complex (e.g. in a graph) $K$ is a set $C$ of $k$-faces such that every $(k-1)$-face is contained in an even number of faces from $C$.} 
 
\begin{lemma}\label{l:genegn} 
(a) Any $k$-cycle in $[n]^{*k+1}$ is a sum of some $k$-octahedra. 
 
(b) Any linear relation on $k$-octahedra is a sum of some relations 
$$K*\{u,v\}+K*\{v,w\}+K*\{w,u\}=0$$
for pairwise distinct $u,v,w\in[n]$, and $(k-1)$-octahedra $K$, and relations obtained from that by permutations of `join coordinates'\footnote{By $*$ we denote a join. 
Recall that a $(k-1)$-octahedron is the join $K=K_0*\ldots*K_{k-1}$ of some 2-element subsets $K_0,\ldots, K_{k-1}\subset[n]$. 
The `join coordinates' of $K*\{a,b\}$ are $K_0,\ldots,K_{k-1},\{a,b\}$. 
We take the same permutation of `join coordinates' of $K*\{u,v\}$, of $K*\{v,w\}$, and of $K*\{w,u\}$.}\footnote{Rigorously, 
a \emph{linear relation} is a set of $k$-octahedra such that every $k$-face of $[n]^{*k+1}$
is contained in an even number of $k$-octahedra from this set.}. 
\end{lemma}

{\bf Comment.}
For $k=1$ this lemma states that 

(a) any $1$-cycle in $K_{n,n}$ is a sum of some cycles of length $4$; 
 
(b) any linear relation on such cycles is a sum of some relations 
$au'bv'+au'bw'+av'bw'=0$ and symmetric to them relations $au'bv'+au'cv'+bu'cv'=0$.  

Of these, (a) is clear \talkno{(by Statement \ref{chl})} and is known. 

\medskip 
Lemma \ref{l:genegn} holds by the following `dual' restatement, Lemma \ref{l:rook}. 
We use the following analogue of the duality between the faces of the regular octahedron and the vertices of the cube. 
 
All $k$-faces of the complex~$[n]^{*k+1}$ are in 1--1 correspondence (duality) with vertices of the $(k+1)$-cube~$[n]^{k+1}$. 
A \textbf{parallelepiped} is a subset $P_0\times\ldots\times P_k\subset [n]^{k+1}$, where $P_i$ are 2-element subsets of $[n]$. 
This is dual to a $k$-octahedron.   

A~\emph{coordinate line} is a subset of~$[n]^{k+1}$ given by fixing all coordinates except one. 
Two $k$-faces of~$[n]^{*k+1}$ have a common face if and only if the two corresponding vertices of~$[n]^{k+1}$ belong to one coordinate line. 
A \textbf{rook cycle} is a~subset of~$[n]^{k+1}$ containing an even number of vertices from every coordinate line. 
This is dual to a $k$-cycle in~$[n]^{*k+1}$.  
The property of being a rook cycle is preserved under modulo~$2$ summation of subsets. 

\begin{lemma}\label{l:rook} (a) Any rook cycle is a sum of some parallelepipeds.
 
(b) Any linear relation on parallelepipeds in $[n]^{k+1}$ is a sum of some relations 
$$P\times\{a,b\}+P\times\{b,c\}+P\times\{c,a\} =0$$ 
for pairwise distinct $a,b,c\in[n]$, and parallelepipeds $P\subset[n]^k$, and relations obtained from that by permutations of coordinates\footnote{\label{f:rela} Rigorously, a \emph{linear relation} is a set of parallelepipeds such that every vertex of $[n]^{k+1}$
is contained in an even number of parallelepipeds from this set.}.
\end{lemma}

\begin{proof}
For $a\in[n-1]^{k+1}$ denote by $P(a):=\{n,a_0\}\times\ldots\times\{n,a_k\}$ the parallelepiped with opposite vertices $a$ and $(n,\ldots,n)$.

For (a) it suffices to prove that {\it any rook cycle $C\subset[n]^{k+1}$ equals to the sum $P(C)$ of parallelepipeds $P(a)$ 
over $a\in C\cap[n-1]^{k+1}$.}  
Let us prove this. 
The sum $C+P(C)$ is a rook cycle. 
Since $P(a)\cap[n-1]^{k+1}=\{a\}$, we have $(C+P(C))\cap[n-1]^{k+1}=\varnothing$.
Together with any vertex $c$ whose $j$-th coordinate is~$n$, the sum $C+P(C)$ contains a vertex $c'$ having the same coordinates except the $j$-th coordinate, and having fewer coordinates equal to $n$ than~$c$.
Hence if non-empty, $C+P(C)$ contains an element from $[n-1]^{k+1}$. 
Thus $C+P(C)$ is empty, i.e., $C=P(C)$. 

Part (b) is the following statement for $r=k+1$. 
For any parallelepiped $P=P_0\times\ldots\times P_k$ define the \emph{complexity} 
to be the number of indices $j$ such that $P_j\subset[n-1]$.
Recall footnote \ref{f:rela}. 

\emph{Every linear relation on parallelepipeds of complexities at most $r$ is a sum of some relations from the conclusion of (b)}. 


Let us prove this statement by induction on $r$.

If $r=0$, then there are pairwise distinct vertices $a^1,\ldots,a^s\in[n-1]^{k+1}$ such that the given relation is $P(a^1)+\ldots+P(a^s)=0$.
The relation has no summands because 
$$\varnothing=(P(a^1)+\ldots+P(a^s))\cap[n-1]^{k+1}=\{a^1,\ldots,a^s\}.$$ 
Here the second equality holds because $P(a)\cap[n-1]^{k+1}=\{a\}$. 

Let us prove the inductive step $r\to r+1$. 

Let $P=P_0\times\ldots\times P_k$ be a parallelepiped such that $P_j\subset[n-1]$ for some $j$. 
Take $a,b\in[n-1]$ such that $P_j=\{a,b\}$. 
Let $P^a$ and $P^b$ be the parallelepipeds obtained from $P$ by replacing $P_j$ with $\{a,n\}$ and with $\{b,n\}$, respectively.
Then one of the relations from the conclusion of (b) is $\{P,P^a,P^b\}$ (written in the form $P+P^a+P^b=0$).
The complexities of $P^a$ and $P^b$ are smaller than the complexity of $P$. 

Consider any linear relation on parallelepipeds of complexities at most $r+1$. 
As in the previous paragraph, every parallelepiped of complexity $r+1$ can be replaced with parallelepipeds of complexities at most $r$, by adding a relation from the conclusion of (b).
Then by the inductive hypothesis the linear relation is a sum of some relations from the conclusion of (b). 
\end{proof}


Let $[n]^{*2}_{\Delta}$ 
be the graph obtained from $K_{n,n}$ by deleting all `diagonal' edges $jj'$, $j\in[n]$.  
E.g. $[3]^{*2}_{\Delta}$ 
is the cycle $12'31'23'$ of length $6$.   
(See the sense of the subscript in a notation of a deleted join in \cite[\S5]{Ma03}.)

\begin{lemma}\label{l:tkn4} 
If $n\ge4$, then any $1$-cycle in $[n]^{*2}_{\Delta}$ 
is a sum of some cycles of length $4$.
\end{lemma}

For a proof 
we need the following simple (and so known) statement. 

\begin{statement}[simple and known]\label{chl} 
Any $1$-cycle in a graph is a sum of some simple cycles for which there are no \emph{chords}, i.e., edges in the graph between two non-consecutive vertices in the cycle.
\end{statement}
 
\begin{proof} Any 1-cycle is a sum of some simple cycles. 
Any chord for a~simple cycle of length~$l$ gives a representation of the cycle as the sum of two simple cycles of lengths less than~$l$.
\end{proof}
  
\begin{proof}[Proof of Lemma \ref{l:tkn4}]
A~simple cycle in $[n]^{*2}_{\Delta}$ 
is formed by an $l$-element cyclic sequence of vertices such that 

$\bullet$  the number~$l$ is even and greater than~$2$; 

$\bullet$ the parts of vertices alternate in the sequence; 

$\bullet$ within each separate part, the vertices are not repeated; 

$\bullet$ there are no two consecutive vertices $m$ and $m'$.

By Statement \ref{chl} we may assume that given 1-cycle is a simple cycle for which there are no chords. 
Then any two non-consecutive vertices are not adjacent, i.e., they either are in the same part or are $m$ and $m'$. 
Any vertex has $\frac{l}{2}-2$ non-consecutive vertices from the other part. 
Then $\frac{l}{2}-2\le1$, so $l\in\{4,6\}$. 

\begin{figure}[h]\centering
\includegraphics[width=2cm]{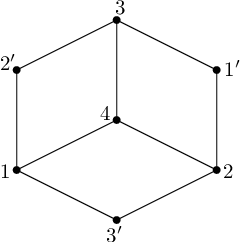}
\caption{A cycle of length 6 is the sum of three cycles of length 4}\label{f:cycles64}
\end{figure}

In a simple cycle of length~$6$ having no chords, any two opposite vertices are $m$ and $m'$. 
Hence the cycle is $(m_1,m'_2,m_3,m'_1,m_2,m'_3)$ for some pairwise distinct $m_1,m_2,m_3$. 
Since $n\ge4$, there is $a\in[n]-\{m_1,m_2,m_3\}$. 
Now the lemma holds because (Figure \ref{f:cycles64})  
$$(m_1,m_2',m_3,a')+(m_2,m_3',m_1,a')+(m_3,m_1',m_2,a').$$ 
\end{proof}

Let $t:[n]^{*2}_{\Delta}\to[n]^{*2}_{\Delta}$ 
be the involution switching the parts, i.e., switching $j$ and $j'$ for every $j\in[n]$.  
Denote by $t$ also the induced involution on the set $H_1([n]^{*2}_{\Delta})$ 
of all $1$-cycles in $[n]^{*2}_{\Delta}$.

\begin{lemma}\label{l:bas} The linear space $H_1([n]^{*2}_{\Delta})$ 
has a basis consisting of (the $t$-symmetric $1$-cycle) $[3]^{*2}_{\Delta}$ 
and pairs of $1$-cycles going to each other under $t$.  
\end{lemma}

\begin{proof} 
Let $S$ be the set of all edges $ij'$ of $[n]^{*2}_{\Delta}$ 
such that $i,j>1$ and $(i,j)\ne(3,2)$.
For $ij'\in S$ let 
$$\widehat{ij'} := \begin{cases} 12'31'ij' &i\ne3,j\ne2\\ 12'3j' &i=3\\ 2'31'i &j=2 \end{cases}.$$ 
Then 

$\bullet$ an edge $e'\in S$ is contained in the cycle $\widehat e$ if and only if $e=e'$; 

$\bullet$ the graph $[n]^{*2}_{\Delta}-S$ 
obtained from $[n]^{*2}_{\Delta}$ 
by deleting all edges of $S$ is a~tree. 
 
The second of these assertions follows because the edges of $[n]^{*2}_{\Delta}-S$ 
are exactly $32'$ and 
$1i',i1'$, where $i>1$, so this graph is obtained by connecting with edge $32'$ the `star' trees consisting of edges $1i'$ with $i>1$, and of edges $1'i$ with $i>1$. 

Then the cycles $\widehat{ij'}$, $ij'\in S$, form a basis of $H_1([n]^{*2}_{\Delta})$.

We have 
\[
t\widehat{ij'} = 1i'j1'23' = 1i'j1'32'12'31'23' = \widehat{23'}+\widehat{ji'}.
\]
Then the required basis is 
$$[3]^{*2}_{\Delta}=\widehat{23'},\quad \widehat{ij'},\quad t\widehat{ij'},\quad\text{where}\quad 
i>j>1 \quad\text{and}\quad (i,j)\ne(3,2).$$
\end{proof}

\begin{remark}\label{dif2}
Lemmas \ref{l:genegn} and \ref{l:rook} generalize to joins of finite sets of different sizes $n_0,\ldots,n_k\ge3$. 
The statements generalize as follows:

$\bullet$ the joinpower $[n]^{*k+1}$ is replaced with $[n_0]*\ldots*[n_k]$,

$\bullet$ the cube $[n]^{k+1}$ is replaced with the parallelepiped $[n_0]\times\ldots\times[n_k]$,

$\bullet$ the numbers $u,v,w,a,b,c$ of parts (b) should belong to $[n_k]$ not to $[n]$, and in the 'permuted' relations these numbers should belong to the corresponding $[n_j]$.

The proofs generalize by replacing 

$\bullet$ $[n-1]^{k+1}$ with $[n_0-1]\times\ldots\times[n_k-1]$,

$\bullet$ $\{n,a_0\}\times\ldots\times\{n,a_k\}$ with $\{n_0,a_0\}\times\ldots\times\{n_k,a_k\}$,

$\bullet$ $(n,\ldots,n)$ with $(n_0,\ldots,n_k)$,

$\bullet$ the phrase 'whose $j$-th coordinate is~$n$' with the phrase 'whose $j$-th coordinate is~$n_j$',

$\bullet$ the phrase 'having fewer coordinates equal to $n$' with the phrase 'having $j$-th coordinates equal to $n_j$ for fewer $j$'s',

$\bullet$ $P_j\subset[n-1]$ with $P_j\subset[n_j-1]$,

$\bullet$ $a,b\in[n-1]$ with $a,b\in[n_j-1]$,

$\bullet$ $\{a,n\}$ with $\{a,n_j\}$,

$\bullet$ $\{b,n\}$ with $\{b,n_j\}$.
\end{remark}

\talkonly{\newpage}
\section{Generators in the homology of the deleted product}\label{s:gendel}

Of this section, only Theorems \ref{r:dedu} and \ref{l:h2sym} are used in other sections. 
  
In this paper $K$ is a $k$-complex (for $k=1$, a graph). 
Denote by $E_k=E_k(K)$ the set of all closed $k$-faces of $K$. 
We consider the standard cell structure on $K^2$, for which $E_k^2$ is the set of all closed $2k$-faces (=cells) of $K^2$.
 
A \talkno{(cellular)} \textbf{$2k$-cycle} in $K^2$ is a subset $C\subset E_k^2$ such that every $(2k-1)$-face of $K^2$ is contained in an even number of $2k$-faces from $C$. 
In other words, such that for every $(k-1)$-face $a$ and $k$-face $\beta$ of $K$ 

$\bullet$ there is an even number of $k$-faces $\alpha$ of $K$ such that $\alpha\supset a$ and $(\alpha,\beta)\in C$, and 

$\bullet$ there is an even number of $k$-faces $\alpha$ of $K$ such that $\alpha\supset a$ and $(\beta,\alpha)\in C$. 
 
Consider the symmetry (involution) $s$ of $K^2$ switching the factors, and the corresponding symmetry of $2k$-cycles in $K^2$.   
A~$2k$-cycle is called \emph{symmetric} if $s$ preserves it.
Let 
$$K^{\times2}_\Delta := \{(\sigma,\tau)\in E_k^2\ :\ \sigma\cap\tau=\varnothing\}\quad\text{and}\quad 
U_{n,k}:=([n]^{*k+1})^{\times2}_\Delta$$ 
be the \emph{combinatorial deleted product} of a complex $K$, and of joinpower.     

For a graph $L$ homeomorphic to $K_5$ or $K_{3,3}$  

$\bullet$ and an edge $\sigma$ of $L$ denote by $[\sigma]$ the edge of $K_5$ or $K_{3,3}$ `containing' $\sigma$; 

$\bullet$ let the \emph{economic deleted product} of $L$ be  
$\{(\sigma,\tau)\in E_1(L)^2\ :\ [\sigma]\cap[\tau]=\varnothing\}.$ 

\begin{theorem}[{\cite[Theorem 4]{Ho07}}, cf. \aronly{\S\ref{s:gendelg}}\jonly{\cite[\S6]{SS23}}]\label{r:dedu}  
For a graph $K$ any symmetric $2$-cycle in $K^{\times2}_\Delta$ is a sum of some of the following ones: 

$\bullet$ $P\times Q+Q\times P$ for vertex-disjoint cycles $P,Q$, 

$\bullet$ the economic deleted product of a subgraph homeomorphic to $K_5$ or to $K_{3,3}$. 
\end{theorem}

\begin{theorem}\label{l:h2sym} 
(a) The following are $s$-symmetric $2k$-cycles in $U_{n,k}$:

$\bullet$ `symmetrized tori' $P\times Q+Q\times P$ for vertex-disjoint $k$-octahedra $P,Q$; 

$\bullet$ $U_{3,k}$. 


(b) Any $s$-symmetric $2k$-cycle in $U_{n,k}$ is a sum of some $2k$-cycles described in (a).  
\end{theorem}
  
\begin{remark}\label{r:gra} 
(a) For $k=1$ Theorem \ref{l:h2sym}.b states that \emph{any symmetric $2$-cycle in $U_{n,1}=(K_{n,n})^{\times2}_\Delta$ is a sum of some of the following ones: $U_{3,1}=(K_{3,3})^{\times2}_\Delta$, and `symmetrized tori' $P\times Q+Q\times P$ for vertex-disjoint cycles $P,Q$ of length $4$ in $K_{n,n}$.}  
Cf. 
Theorem \ref{r:dedu}. 
 
(b) Proof of Theorem \ref{l:h2sym}.b provides a simpler proof (and a higher-dimensional generalization) of \cite[the formula for $b_2(F(\Gamma,2))$ in Corollary 24]{FH10}: $\dim H_{2k}(U_{n,k}) =  (n^2-3n+1)^{k+1}$ for $n\ge3$.

\end{remark}
 

In this paper, the following operations are considered for subspaces in vector spaces over $\Z_2$: linear spans ($\langle\dots\rangle$), direct sums ($\oplus$), tensor products ($\otimes$) and $\ell$-th tensor powers ($\dots^{\otimes \ell}$).

\newcommand{\hag}[1]{\left\langle#1\right\rangle}

Cellular $j$-cycles in $K^j$ are defined analogously to $j=2$.  
Denote by $H_j(X)$ the group of (cellular) $j$-cycles in a (cell) complex $X$ of dimension $j$.   

\begin{theorem}\label{p:symgrak} (a) There is an isomorphism    
$$\psi : H_{2k}(U_{n,k}) \to H_1([n]^{*2}_{\Delta})^{\otimes k+1} \quad\text{such that}\quad \psi s = t^{\otimes k+1} \psi, \quad U_{3,k} = \psi^{-1}(([3]^{*2}_{\Delta})^{\otimes k+1}),$$ 
and the tensor product of any $k+1$ cycles of length 4 is the $\psi$-image of the product of some two disjoint $k$-octahedra. 

(b) We have 
$$\Ker(I-t^{\otimes \ell}) = \Img(I+t^{\otimes \ell})\oplus\hag{([3]^{*2}_{\Delta})^{\otimes \ell}}\subset 
H_1([n]^{*2}_{\Delta})^{\otimes \ell}.$$ 
\end{theorem}

\begin{proof}[Proof of Theorem \ref{l:h2sym}]
(a) This follows because any symmetrized torus is a symmetric $2k$-cycle and by Theorem \ref{p:symgrak}.a. 

(b) The elements of $H_1([n]^{*2}_{\Delta})^{\otimes k+1}$ 
invariant under the involution $t^{\otimes k+1}$ form the subgroup 
$\Ker(I-t^{\otimes k+1})$.  
By Lemma \ref{l:tkn4} $\Img(I+t^{\otimes k+1})$ is generated by the `symmetrized tori' 
$$a_0\otimes\ldots\otimes a_k+t(a_0)\otimes\ldots\otimes t(a_k) \in H_1([n]^{*2}_{\Delta})^{\otimes k+1},$$ 
for cycles $a_0,\ldots,a_k$ of length $4$ in $[n]^{*2}_{\Delta}$.
Thus (b) follows by Theorem \ref{p:symgrak}.a and Theorem \ref{p:symgrak}.b for $\ell=k+1$. 
\end{proof}

\begin{proof}[Proof of Theorem \ref{p:symgrak}.a]
Vectors $\alpha,\beta\in[n]^{k+1}$ are said to be \emph{disjoint} if $\alpha_j\ne\beta_j$ for every $j=0,\ldots,k$. 
For a set $X$, the set $2^X$ is considered with the operation of a sum modulo $2$.
Let $\psi : =\varkappa|_{\ldots} \varphi|_{\ldots}$ be the composition of the isomorphisms from the following commutative diagram: 
$$\xymatrix{
H_{2k}(U_{n,k}) \ar[r]^{\varphi|_{\ldots}} \ar@{^{(}->}[d] & H_{k+1}(([n]^{*2}_{\Delta})^{\times k+1}) 
\ar[r]^{\varkappa|_{\ldots}} \ar@{^{(}->}[d] &  H_1([n]^{*2}_{\Delta})^{\otimes k+1} \ar@{^{(}->}[d] \\
2^{E_{2k}(U_{n,k})} \ar[r]^\varphi \ar@/_1.7pc/[rr]^{\psi} 
& 2^{E_{k+1}(([n]^{*2}_{\Delta})^{\times k+1})} \ar[r]^\varkappa & (2^{E_1([n]^{*2}_{\Delta})})^{\otimes k+1}
}, \quad\text{where}
$$




$\bullet$ $\varkappa$ is the K\"unneth isomorphism that maps the product
$y_0\times\ldots\times y_k$ to $y_0\otimes\ldots\otimes y_k$; 

$\bullet$ $\varphi$ is the isomorphism induced by the 1--1 correspondence\footnote{Recall that for $2$-element subsets $P_0,\ldots,P_k\subset[n]$ the $k$-faces $a_0*\ldots*a_k$, $a_i\in P_i$, 
of the $k$-octahedron $P_0*\ldots*P_k$ are spanned by vertices $(i,a_i)$.}
$$\varphi' : E_{2k}(U_{n,k}) \to E_{k+1}(([n]^{*2}_{\Delta})^{\times k+1}) \text{ given by } 
\varphi'\bigl((\alpha_0*\ldots*\alpha_k)\times(\beta_0*\ldots*\beta_k)\bigr) =(\alpha_0*\beta_0)\times\ldots\times(\alpha_k*\beta_k)$$
between $2k$-cells and $(k+1)$-cells for disjoint $\alpha,\beta\in[n]^{k+1}$. 

If one of $\alpha_0,\ldots,\alpha_k,\beta_0,\ldots,\beta_k$ is replaced by the empty set, then the formula for $\varphi'$ gives a 1--1 correspondence between $(2k-1)$-cells of $U_{n,k}$ and $k$-cells of $([n]^{*2}_{\Delta})^{\times k+1}$. 
The two 1--1 correspondences respect incidence.  
Thus indeed $\varphi H_{2k}(U_{n,k}) = H_{k+1}(([n]^{*2}_{\Delta})^{\times k+1})$. 
 
Then for any disjoint $\alpha,\beta\in[n]^{k+1}$, we have 
\begin{equation}\label{rule}
\psi\bigl((\alpha_0*\ldots*\alpha_k)\times(\beta_0*\ldots*\beta_k)\bigr) = (\alpha_0*\beta_0)\otimes\ldots\otimes(\alpha_k*\beta_k).
\end{equation}

The equality $\psi s=t^{\otimes k+1}\psi$ holds because it is easily verified on generators:
\begin{gather*}
\psi s\bigl((\alpha_0*\ldots*\alpha_k)\times(\beta_0*\ldots*\beta_k)\bigr)=
\psi\bigl((\beta_0*\ldots*\beta_k)\times(\alpha_0*\ldots*\alpha_k)\bigr)=\\
=(\beta_0*\alpha_0)\otimes\ldots\otimes(\beta_k*\alpha_k)=(t(\alpha_0*\beta_0))\otimes\ldots\otimes(t(\alpha_k*\beta_k))=\\
=t^{\otimes k+1}(\alpha_0*\beta_0)\otimes\ldots\otimes(\alpha_k*\beta_k)=
t^{\otimes k+1}\psi\bigl((\alpha_0*\ldots*\alpha_k)\times(\beta_0*\ldots*\beta_k)\bigr).
\end{gather*}

Summing up \eqref{rule} over all disjoint $\alpha,\beta\in[3]^{k+1}$ gives $\psi(U_{3,k})=([3]^{*2}_{\Delta})^{\otimes k+1}$. 

For any cycle $a_j$ of length 4 in $[n]^{*2}_{\Delta}$, 
$j=0,\ldots,k$, there are disjoint $2$-element subsets $P_j,Q_j\subset [n]$ such that $a_j=P_j*Q_j$. 
Then $P=P_0*\ldots*P_k$ and $Q=Q_0*\ldots*Q_k$ are disjoint $k$-octahedra in $[n]^{*k+1}$.  
Summing up \eqref{rule} over all $\alpha_j\in P_j$ and $\beta_j\in Q_j$ gives $\psi(P\times Q)=a_0\otimes\ldots\otimes a_k$. 
\end{proof}

\begin{proof}[Proof of Theorem \ref{p:symgrak}.b]
Take tensor $\ell$-fold products of 1-cycles of~a basis from Lemma \ref{l:bas}. 
They form a~basis 
$$y_{-m},\ldots,y_{-1},\ y_0=([3]^{*2}_{\Delta})^{\otimes \ell},\ y_1,\ldots,y_m\quad\text{of}\quad 
H_1([n]^{*2}_{\Delta})^{\otimes \ell}\quad\text{such that}\quad t^{\otimes \ell}y_i=y_{-i}.$$
(I.e. the basis is involutively permuted by~$t^{\otimes \ell}$, with the only fixed vector~$y_0$.) 

Thus the subgroups $\Ker(I-t^{\otimes \ell})$ and $\Img(I+t^{\otimes \ell})$ have bases  
$$y_0,y_1+y_{-1},\ldots,y_m+y_{-m}\quad\text{and}\quad y_1+y_{-1},\ldots,y_m+y_{-m},$$ 
respectively. 
This implies Theorem \ref{p:symgrak}.b. 
\end{proof}

\begin{remark}\label{dif3}
Theorems \ref{l:h2sym} and \ref{p:symgrak} generalize to joins of finite sets of different sizes $n_0,\ldots,n_k\ge3$. 

The statements generalize by replacing (below $r=k+1,\ell$) 

$\bullet$ $U_{n,k}$ with $([n_0]*\ldots*[n_k])^{\times2}_\Delta$,

$\bullet$ $H_1([n]^{*2}_{\Delta})^{\otimes r}$ with $H_1([n_0]^{*2}_{\Delta})\otimes\ldots\otimes H_1([n_{r-1}]^{*2}_{\Delta})$,

$\bullet$ $t$ with the corresponding involutions $t_j$ of $[n_j]^{*2}_{\Delta}$, of $2^{E_1([n_j]^{*2}_{\Delta})}$, and of $H_1([n_j]^{*2}_{\Delta})$,


$\bullet$ $t^{\otimes r}$ with $t_0\otimes\ldots\otimes t_{r-1}$. 

The proofs generalize by replacing 

$\bullet$ $[n]^{k+1}$ with $[n_0]\times\ldots\times[n_k]$,

$\bullet$ $([n]^{*2}_{\Delta})^{\times k+1}$ with $[n_0]^{*2}_{\Delta}\times\ldots\times[n_k]^{*2}_{\Delta}$,

$\bullet$ $(2^{E_1([n]^{*2}_{\Delta})})^{\otimes k+1}$ 
with $2^{E_1([n_0]^{*2}_{\Delta})} \otimes\ldots\otimes 2^{E_1([n_k]^{*2}_{\Delta})}$,

$\bullet$ $[n]^{*k+1}$ with $[n_0]*\ldots*[n_k]$,

\aronly{$\bullet$ the phrase 'for cycles $a_0,\ldots,a_k$ of length $4$ in $[n]^{*2}_{\Delta}$' 
with the phrase 'for cycles $a_0,\ldots,a_k$ of length $4$ in $[n_0]^{*2}_{\Delta},\ldots,[n_k]^{*2}_{\Delta}$ 
respectively',}

$\bullet$ the phrase 
'For any cycle $a_j$ of length 4 in $[n]^{*2}_{\Delta}$, 
$j=0,\ldots,k$, there are disjoint $2$-element subsets $P_j,Q_j\subset [n]$' 
with the phrase 
'For any cycle $a_j$ of length 4 in $[n_j]^{*2}_{\Delta}$, 
$j=0,\ldots,k$, there are disjoint $2$-element subsets $P_j,Q_j\subset [n_j]$'.
\end{remark}

\talkonly{\newpage}
\section{Proof of the `if' part of Theorem \ref{t:embntfk}}\label{s:meta}

In this section $H,V$ are any linear spaces over $\Z_2$, and $\cdot:H\times H\to\Z_2$ is any  symmetric bilinear form. 

\begin{lemma}\label{l:homom}
Let $P_1,\ldots,P_p$ be generators of $V$, and $Y_1,\ldots,Y_p\in H$ be vectors.  
Assume that if a sum of some $P_j$ is zero, then the $\cdot$-product of the sum of the corresponding $Y_j$, and $Y_i$, is zero for any $i\in[p]$. 

Then there is a homomorphism $\psi:V\to H$ such that $\psi P_a\cdot\psi P_b=Y_a\cdot Y_b$ for every $a,b\in[p]$.
\end{lemma}

\begin{proof}
\footnote{Observe that the map $P_a\mapsto Y_a$ need not extend to a homomorphism $V\to H$.}
Without loss of generality, $P_1,\ldots,P_t$ is a basis of $V$. 
Define a homomorphism $\psi:V\to H$ by setting $\psi(P_j):=Y_j$ for every $j\in[t]$. 

Express any $P_a$ as a sum of some basic elements: $P_a=\sum_j P_j$. 
Then $\psi P_a-Y_a = \sum_j Y_j-Y_a$ 

$\bullet$ is a sum of some of vectors $Y_1,\ldots,Y_p$;  

$\bullet$ has zero $\cdot$-product with $Y_i$ for any $i\in[p]$ (by the assumption).  
 
Hence 
$$\psi P_a\cdot\psi P_b - Y_a\cdot Y_b = 
(\psi P_a-Y_a)\cdot Y_b + Y_a\cdot(\psi P_b-Y_b) + (\psi P_a-Y_a)\cdot(\psi P_b-Y_b) = 0.$$ 
\end{proof} 

For a map $y:E_k\to H$ define 

$\bullet$ a homomorphism $\widehat y:H_k(K)\to H$ by $\widehat y(A) := \sum\limits_{\sigma\in A} y(\sigma)$; 

$\bullet$ $y^2(C,\cdot)  := \sum\limits_{\{\sigma,\tau\} \in C/s} y(\sigma)\cdot y(\tau)\in\Z_2$, where $C$ is a symmetric $2k$-cycle in $K^2$. 


\begin{lemma}\label{l:evalc}  (a) For any homomorphism $\psi:H_k(K)\to H$ there is a map $y:E_k\to H$ such that $\psi=\widehat y$. 

(b) Let $y:E_k\to H$ be a map.  
For $k$-cycles $P_1,\ldots,P_q$, $Q_1,\ldots,Q_q$ in $K$ take a symmetric $2k$-cycle  $C_{P,Q} := \sum\limits_{j=1}^q(P_j\times Q_j+Q_j\times P_j)$  in $K^2$.    
Then $y^2(C_{P,Q},\cdot) =  \sum\limits_{j=1}^q \widehat y(P_j)\cdot\widehat y(Q_j)$.   
\end{lemma}


\begin{proof}[Proof of (a)]
(Part (a) is simple and known.)  
The linear operator $\psi$ extends to a linear operator $\t{y}:2^{E_k}\to H$. 
Define the map $y:E_k\to H$ by $y(\sigma):=\t{y}(\{\sigma\})$.  
So if $A\in H_k(K)$, then $\widehat y(A) = \sum\limits_{\sigma\in A} y(\sigma) = \sum\limits_{\sigma\in A} \t{y}(\{\sigma\}) = \t{y}(A) = \psi(A)$. 
\end{proof}

Part (b) is proved by direct calculation.  

Denote by $|S|_2\in\Z_2$ the number of elements modulo 2 in a finite set $S$. 

Let $f:K\to\R^{2k}$ be a general position PL map (see definition in \cite[Definition 1.1.6.a]{Sk24}).  
For a symmetric $2k$-cycle $C$ in $K^{\times2}_\Delta$ let the {\it $C$-van Kampen number of $K$} be 
$$v(C)=v_K(C) := \sum\limits_{\{\sigma,\tau\}\in C/s} |f\sigma\cap f\tau|_2\in\Z_2,$$
Recall that $v_K(C)$ depends only on $K,C$, not on $f$, cf. \cite[Remark 1.5.5.c]{Sk18}.
 

\begin{theorem}[E. Kogan; {\cite[(E)$\Leftrightarrow$(R') of Proposition 2.5.1]{KS21e}}]\label{t:homembcri} 
Let $M$ be a closed $(k-1)$-connected $2k$-manifold. 
There is a $\Z_2$-embedding $K\to M$ if and only if 

(R') there is a map $y:E_k\to H_k(M)$ such that for any symmetric $2k$-cycle $C$ in $K^{\times2}_\Delta$ we have $v_K(C) = y^2(C,\cap_M)$. 
\end{theorem}

Theorem \ref{t:homembcri} follows because (as we prove below) (R') is equivalent to the condition (RH) of Theorem \ref{t:cohomembcri} below. 

\talkonly{\newpage}
\begin{proof}[Proof of the `if' part of Theorem \ref{t:embntfk}]
Let $V=H_k([n]^{*k+1})$. 
In this paragraph we prove that there is a homomorphism $\psi:V\to H_k(M)$ such that 
$$(*)\qquad\psi P\cap_M\psi Q=(Y^T\Omega_MY)_{P,Q}\quad\text{for any $k$-octahedra}\quad P,Q.$$
Take $p={n\choose2}^{k+1}$ and all different $k$-octahedra $P_1,\ldots,P_p\in V$. 
By Lemma \ref{l:genegn}.a they generate $V$. 
Take a basis of $H_k(M)$ in which $\cap_M$ has matrix $\Omega_M$. 
For every $a\in[p]$ let $Y_a\in H_k(M)$ be the vector expressed in this basis as the column of the given matrix $Y$ corresponding to $P_a$.  
Then for any $a,b\in[p]$ we have $Y_a\cap_M Y_b=(Y^T\Omega_MY)_{a,b}$. 
So by Lemma \ref{l:genegn}.b and the additivity of $Y^T\Omega_MY$ the assumption of Lemma \ref{l:homom} is fulfilled for $H=H_k(M)$ and $\cdot=\cap_M$. 
This gives the required homomorphism $\psi$. 
 
By Lemma \ref{l:evalc}.a there is a map $y:E_k\to H_k(M)$  such that $\psi=\widehat y$.
Both $y^2(C,\cap_M)$ and $v(C)$ are linear in $C$. 
So by Theorems \ref{t:homembcri} and \ref{l:h2sym}.b it suffices to check the equality $y^2(C,\cap_M)=v(C)$ only for $2k$-cycles $C$ from Theorem \ref{l:h2sym}.a. 

Let $C=P\times Q+Q\times P$ for vertex-disjoint $k$-octahedra $P,Q$. 
Then  
$$y^2(C,\cap_M) \overset{(1)}= \psi(P)\cap_M\psi(Q) \overset{(2)}= 0 \overset{(3)}= |f(P)\cap f(Q)|_2 \overset{(4)}= v(C),\quad\text{where}$$  

$\bullet$ (1) holds by Lemma \ref{l:evalc}.b; 

$\bullet$ (2) holds by (*) and the independence; 

$\bullet$ (3) holds by the well-known Parity Lemma\talkno{ (asserting that every two general position $k$-cycles in $\R^{2k}$ intersect at an even number of points, cf. \cite[\S1.3, \S4.8 `Algebraic intersection number', and Lemma~5.3.4]{Sk})};  


$\bullet$ (4) holds by the definition of $v(C)$.  

Let $C=U_{3,k}$. 
Then $C=\sum(P\times Q+Q\times P)$ is the sum over all $2^k$ unordered pairs $\{P,Q\}$ of $k$-octahedra $P,Q\subset [3]^{*k+1}$ whose only common $k$-face is $1^{*k+1}$ \talkno{\cite[Lemma 2.3]{DS22}}\talkonly{\cite{DS22}}.  
Hence  
$$y^2(C,\cap_M) \overset{(1)}= \sum
\psi(P)\cap_M\psi(Q) \overset{(2)}= 
S_{[3]^{*k+1},1^{*k+1}} Y^T\Omega_MY \overset{(3)}=1 \overset{(4)}= v(C),\quad\text{where}$$  

$\bullet$ (1) holds by Lemma \ref{l:evalc}.b; 

$\bullet$ (2) holds by (*) and  
definition of $S_{X,e}$; 

$\bullet$ (3) holds by the non-triviality;  
and 

$\bullet$ (4) holds by \cite[Satz 5]{vK32} (for an alternative modern proof see \cite[\S4, proof of Lemma 4.2]{DS22}).  
\end{proof}

\begin{theorem}\label{t:cohomembcri} 
Let $M$ be a closed $(k-1)$-connected $2k$-manifold. 
There is a $\Z_2$-embedding $K\to M$ if and only if 

(RH) there is a homomorphism $\psi:H_k(K)\to H_k(M)$ such that $\omega(\psi)=v_K$ (see the required definitions below).
\end{theorem}

This is \cite[Proposition 21 and Theorem  15]{PT19}, \cite[Theorem 1.1.5 and Lemma 2.1.1]{Sk24}. 

Let $f:K\to\R^{2k}$ be a general position PL map. 
Take any pair of vertex-disjoint $k$-faces $\sigma,\tau$ of $K$.
By general position the intersection $f\sigma\cap f\tau$ consists of a finite number of points.
Assign to the pair $\{\sigma,\tau\}$ the residue
$$\nu(f)\{\sigma,\tau\}:=|f\sigma\cap f\tau|_2.$$
Denote by $K^*:=K^{\times2}_\Delta/s$ the set of all unordered pairs of non-adjacent 
$k$-faces of $K$.

The obtained map $\nu(f):K^*\to\Z_2$ is called the (modulo 2) {\bf intersection cocycle} of $f$. 

We consider the standard cell structure on $K^*$. 
Let $H^{2k}(K^*)=H^{2k}_{sym}(K^{\times2}_\Delta;\Z_2)$ be the group of modulo 2 cohomology classes of maps (=cocycles) $K^*\to\Z_2$.

It is classical \cite{vK32} 
that the class 
$$v_K :=[\nu(f)] \in H^{2k}(K^*)$$ 
is well-defined, i.e. is independent of the choice of $f$.

For a map $y:E_k\to H_k(M)$ define a map 
$$\omega(y):K^*\to\Z_2 \quad\text{by}\quad \omega(y)\{\sigma,\tau\} := y(\sigma)\cap_M y(\tau).$$
By Lemma \ref{l:evalc}.a for any homomorphism $\psi:H_k(K)\to H_k(M)$ there is a map $y:E_k\to H_k(M)$  such that $\psi=\widehat y$.
Define 
$$\omega(\psi) := [\omega(y)]\in H^{2k}(K^*).$$
It is easy to check that $\omega(\psi)$ is well-defined, i.e. that maps $\omega(y)$ are cohomologous for two maps $y$ giving the same homomorphism $\psi=\widehat y$ (because such maps $y$ are cohomologous).

\begin{remark}\label{r:alter} 
 We have 
$$\omega(\psi) := \Bigl((\psi\otimes\psi)^*(\cap_M)\Bigr)|_{K^{\times2}_\Delta} \in H^{2k}_{sym}(K^{\times2}_\Delta;\Z_2), \quad\text{where}$$
 
$\bullet$ $\cap_M$ is considered to be an element of $\Hom(H_k(M)\otimes H_k(M))$.

$\bullet$ we identify the following groups by the Universal Coefficients isomorphism (natural in this situation) and the K\"unneth isomorphism: 
$$\Hom(H_k(K)\otimes H_k(K))\cong H^k(K)\otimes H^k(K)\cong H^{2k}(K\times K).$$

Since $\cap_M$ is symmetric, $\omega(\psi)$  is indeed symmetric.
See \cite[Remark 2.5.2b]{KS21e}, cf. \cite{PT19}.
\end{remark}
  
Denote by $\frown:H^{2k}(K^*)\times H_{2k}(K^*)\to\Z_2$ the `scalar' product. 

\begin{proof}[Proof that (RH)$\Rightarrow$(R')]
Take a general position PL map $f\colon K\to\R^{2k}$. 
By Lemma \ref{l:evalc}.a there is a map $y:E_k\to H_k(M)$ such that $\psi=\widehat y$.
Then by (RH) 
$$y^2(C,\cap_M) = [\omega(y)]\frown C/s = [\nu(f)]\frown C/s = v(C).$$
\end{proof}

\begin{proof}[Proof that (R')$\Rightarrow$(RH)]
Take a general position PL map $f\colon K\to\R^{2k}$. 
Let $\psi=\widehat y$.
Then 
$$\omega(\psi) = [\omega(y)] = [\nu(f)] = v_K.$$
Here the 
middle equality holds because the `scalar' product $\frown$ is nondegenerate, and for any symmetric $2k$-cycle $C$ in $K^{\times2}_\Delta$ by (R') we have
$$\sum_{\{\sigma, \tau\}\in C/s} (\omega(y)-\nu(f))\{\sigma,\tau\} = y^2(C,\cap_M)-v(C) = 0.$$
\end{proof}

\begin{remark}\label{dif4} 
Theorem \ref{t:embntfk} generalizes to joins of finite sets of different sizes $n_0,\ldots,n_k\ge3$ 
(see Remark \ref{r:othcom}.c).
The proof of the `if' part of Theorem \ref{t:embntfk} generalizes by replacing $[n]^{*k+1}$ with $[n_0]*\ldots*[n_k]$ and 
'$p={n\choose2}^{k+1}$' with '$p={n_0\choose2}\ldots{n_k\choose2}$'.
\end{remark}

\aronly{

\section{On the paper \cite{Sa91}}\label{s:gendelg} 


In \cite{Sa91} proof of \cite[3.4.1]{Sa91} (Theorem \ref{r:dedu}) has the following two gaps, see Remarks \ref{gap1}.a and \ref{gap2}.b. 
We use the notation from \cite{Sa91}: $K_*$ is the deleted product ($K^{\times2}_\Delta$ of this paper). 

\begin{remark}\label{gap1} 
For edges $\alpha,\beta$ of $K$ denote by $K_{\alpha,\beta}$ the graph obtained from $K$ by \emph{subdividing} both edges $\alpha,\beta$, and then \emph{identifying} both appearing points of degree 2 to a new vertex denoted by $0$.

(a) In \cite[bottom of p. 86]{Sa91}, the statement 

\quad (*) '\emph{But $L_*$ has less symmetric circuits than $(K_z)_*$}' 

is not proved.
Here 

\qquad (0) `symmetric circuit' is non-standard terminology for `symmetric (cellular) 2-cycle'; 

\qquad (1) $K_z$ is a graph that is the support of a minimal symmetric 2-cycle $z$ in $(K_z)_*$,  and $(K_z)_*$ contains a minimal symmetric 2-cycle $w\ne z$; 

\qquad (2) $\alpha\times\beta$ is a cell from $w-z$, and $L:=(K_z)_{\alpha,\beta}$.

A `trivial proof' of (*) fails by (b). 
Although (b) is not a counterexample to (*), (b) shows that the proof of (*) cannot be trivial because it has to use the assumptions (1) and (2).  

(b) \emph{There is a graph $K$ and its disjoint edges $\alpha,\beta$ such that $K_{\alpha,\beta,*}$ contains a symmetric 2-cycle not contained in $K_*$.}
 
Indeed, let $K$ be the union of the cycle $12345$ and chord $14$. 
In $K_{12,34,*}$ the symmetric 2-cycle $C:=023\times145+145\times023$ contains $02\times14$. 
The cell $02\times14$ is not contained in $K_*$. 
Therefore, $C$ is not contained in $K_*$.\footnote{Alternatively, let $K$ be the union of $K_4$ and the paths $152$ and $364$. 
Then $K$ is planar but $K_{12,34}$ is homeomorphic to $K_5$, so $K_{12,34}$ contains a symmetric 2-cycle not contained in $K_*$.} 


(c) In \cite[\S5, proof of Lemma 5.2]{SSv3}, there is the following gap similar to the one indicated in (a). 
The following statement is not proved (and by (b) is incorrect without quite technical assumptions on the graph $L$ and the edges $e_1,e_2$ for which $L_0=L_{e_1,e_2}$):

`\emph{... symmetric $2$-cycles in~$(L_0)^{\times2}_\Delta$ are exactly symmetric $2$-cycles in~$L^{\times2}_\Delta$ not containing $(e_1,e_2)$.}' 
\end{remark}

\begin{remark}\label{gap2} 
(a) \emph{Lemma.} Define a homomorphism\footnote{The argument \cite[3.3.1 and its proof]{Sa91} only uses that $v$ is a homomorphism trivial on symmetrized tori $c_1\times c_2+c_2\times c_1$. 
However, that argument has a gap, see (b).} $v:H_2^s(K_*;\Z_2)\to\Z_2$ by $v(c):=\left<o(K),c\right>$. 
If $v$ is non-zero, then there is a subgraph $G$ of $K$ such that $G$ is homeomorphic either to $K_5$ or to $K_{3,3}$, and the restriction of $v$ to $H_2^s(G_*;\Z_2)$ is non-zero.

This lemma holds by Kuratowski criterion. 
Observe that \cite[3.3.1 and its proof]{Sa91} intends to \emph{reprove} Kuratowski criterion. 

(b) `\emph{...[The following] results from the above discussion}' before \cite[3.4.1]{Sa91} is unjustified. 
Indeed, `the above discussion' \cite[3.3.1 and its proof]{Sa91} amounts to Lemma (a).   
In order to conclude \cite[3.4.1]{Sa91} from Lemma (a), one needs (an argument relating symmetric and ordinary homology, and) the following statement \cite[after 3.4.2]{Sa91}: 

\emph{The toral circuits generate the kernel of the linear functional
$\left<o(K),\cdot\right>:H_2^s(K_*;\Z_2)\to\Z_2$.} 

This statement is non-trivial, but is presented in \cite[after 3.4.2]{Sa91} without proof or reference. 
(The statement is correct by \cite{Ho07} which proves \cite[3.4.1]{Sa91} itself.
See also Remark \ref{r:samo}.a.) 
 

An attempt to recover this gap is presented in \cite[\S5]{SSv3}, by turning the argument \cite[3.3.1 and its proof]{Sa91} (essentially attempting to prove Lemma (a) independently of Kuratowski criterion) to the proof of (the symmetrized version of) \cite[3.4.1]{Sa91} itself. 
Unfortunately, that attempt reproduced the gap \ref{gap1}.(a) as described in \ref{gap1}.(c). 
\end{remark}

\begin{remark}\label{r:samo}
(a) The following corollary of \cite[3.4.2]{Sa91} is wrong  (even for connected graphs):  

\emph{For any graph $K$ there is a cellular 2-cycle $C$ in $K_*$ such that any cellular 2-cycle in $K_*$ is a sum of some of the following cellular 2-cycles: $C$ and products of vertex-disjoint cycles.} 

Indeed, take disjoint union of two copies of $K_5$.  
(Make the disjoint union connected by joining the copies with an edge.)
Cf. \cite{HNT}. 

(b) For a description and recovery of a gap in \cite[\S2]{Sa91} see  \cite[footnote 6]{Sc13} and \cite[proof of Theorem 3.1]{MPS}. 
\end{remark}

}



\aronly{ 
 
\section{Appendix: some alternative proofs}\label{ss:appeal}

\begin{proof}[A direct proof of Lemma \ref{l:genegn}.b for $k=1$] 
Denote by $\alpha(a,b;u,v,w)$ and $\beta(a,b,c;u,v)$ the relations from the statement.
For any $b,u,v\in[n]-\{1\}$ such that $u\ne v$, one has
\begin{equation*}
\begin{aligned}
&\text{for any } a\in[n]-\{b\}&\quad&au'bv'=a1'bu'+a1'bv'+\alpha(a,b;1,u,v);\\
&\text{for any } a\in[n]-\{1,b\}&\quad&a1'bv'=11'av'+11'bv'+\beta(1,a,b;1,v).
\end{aligned}
\end{equation*}
Hence any linear relation on cycles of length 4 in $K_{n,n}$ can be reduced (by adding some of $\alpha$'s and $\beta$'s) to a relation including only cycles $11'bv'$.  
The latter relation is trivial because for any $a,b,u,v\in[n]-\{1\}$ the edge $\{a,u'\}$ belongs to the cycle $11'bv'$ if and only if $a=b$ and $u=v$.
\end{proof}

\begin{proof}[Alternative proof of Theorem \ref{p:symgrak}.b] 
Induction on $\ell$. 

Let us prove the base $\ell=1$. 
If $n=3$, then $H_1(\t{K_3})=\Ker(I+t)=\{0,\t{K_3}\}$ and $\Img(I+t)=0$. 
Assume further that $n\ge4$. 
Observe that $\t{K_n}/t=K_n$. 
Denote by $q:\t{K_n}\to K_n$ the quotient map. 
Since $t$ has no fixed points, the restriction of the induced map $q_*:H_1(\t{K_n})\to H_1(K_n)$ to the subgroup $\Ker(I+t)\subset H_1(\t{K_n})$ of $t$-symmetric cycles is an isomorphism.  
(Indeed, the inverse map $R:H_1(K_n)\to\Ker(I+t)$ is defined by $R(y):=q^{-1}(y)$.)  
Clearly, $q(\t{K_3})=(1,2,3)$, and for a cycle $m$ of length 4 in $\t{K_n}$ we have $q(m+tm)=q(m)$ is a cycle of length 4 in $K_n$. 
Hence by Lemmas \ref{l:tkn4} and \ref{l:genegn}.a $q\Img(I+t)$ is generated by cycles of length 4 in $K_n$. 

In the following paragraph we prove that \emph{\it any $1$-cycle in $K_n$ is a sum of some cycles of length~$4$ and, possibly, of the cycle $(1,2,3)$.}

The sum of two cycles of length~$3$ having a common edge is a cycle of length~$4$. 
So if a cycle $\alpha$ of length~$3$ is a sum of $(1,2,3)$ and cycles of length 4, then any cycle of length~$3$ having a common edge with $\alpha$ is such. 
Hence all cycles of length~$3$ are such. 
Now the required assertion follows because by Statement \ref{chl} any $1$-cycle in $K_n$ is a sum of some cycles of length 3. 

Then $\Ker(I+t)$ is generated by $\t{K_3}$ and $\Img(I+t)$.  
The number of edges modulo 4 defines a homomorphism $\Ker(I+t)\to\Z_4$. 
This homomorphism is 2 on $\t{K_3}$ and is 0 on $\Img(I+t)$. 
Hence $\t{K_3}\not\in\Img(I+t)\subset H_1(\t{K_n})$. 

 
The inductive step $\ell\to \ell+1$ is obtained by applying the following Lemma \ref{l:pro} to 
$$W_1=H_1(\t{K_n}),\quad W_2=H_1(\t{K_n})^{\otimes \ell},\quad x_1=\t{K_3},\quad x_2=\t{K_3}^{\otimes \ell},\quad A_1=t,\quad A_2=t^{\otimes \ell}.$$
\end{proof}

\begin{lemma}\label{l:pro} For $j=1,2$ let $W_j$ be a vector space over $\Z_2$, \ $x_j\in W_j-\{0\}$ a vector, and $A_j$ an operator in $W_j$ such that $\Ker(I+A_j)=\Img(I+A_j)\oplus\hag{x_j}$. 
Then the operator $A:=A_1\otimes A_2$ in $W:=W_1\otimes W_2$ satisfies $\Ker(I+A)=\Img(I+A)\oplus\hag{x_1\otimes x_2}$.
\end{lemma}

\begin{proof} 
Let $N_j:=I+A_j$. 
Then $\Ker N_j=(\Img N_j)\oplus\hag{x_j}$. 
Hence $N_j^2=0$. 

Denote by $J_m(\lambda)$ the Jordan cell of size~$m$ with eigenvalue~$\lambda$. 
Take the Jordan form of $N_j$, i.e., represent~$W_j$ as a~direct sum of $N_j$-invariant subspaces $W_{j,p}$, $0\le p\le m_j$, such that

$\bullet$  $W_{j,0}=\hag{x_j}$, $N_jW_{j,0}=0$; 

$\bullet$ for every $p>0$, the subspace~$W_{j,p}$ is two-dimensional, and the operator~$N_j$ acts on it with matrix $J_2(0)$ (in some basis).

Hence the operator~$A_j$ also preserves every subspace~$W_{j,p}$, acting on it identically if $p=0$, and with matrix $J_2(1)$ (in some basis) otherwise.
So the space~$W$ is a~direct sum of $A$-invariant subspaces $W_{p_1,p_2}:=W_{1,p_1}\otimes W_{2,p_2}$, $0\le p_j\le m_j$. 
Now the lemma holds because $W_{0,0}=\hag{x_1\otimes x_2}$, $(I+A)W_{0,0}=0$, and  

(*) \qquad $\Ker(I+A|_U)=\Img(I+A|_U)$ on every other summand $U:=W_{p_1,p_2}$. 

Let us prove assertion (*). 
Denote by $M(X)$ the matrix of an operator $X$ in the considered basis. 

If either $p_1>0=p_2$ or $p_2>0=p_1$, then $U$ has a~basis $e_1,e_2$ in which~$M(A|_U)=J_2(1)$. 
So $M(I+A|_U)=J_2(0)$. 
Thus, $\Ker(I+A|_U)=\Img(I+A|_U)=\hag{e_1}$.

\newcommand{\rbmat}[1]{\begin{pmatrix}#1\end{pmatrix}}

If $p_1,p_2>0$, then $U$ has a~basis $e_1,e_2,e_3,e_4$ in which
\begin{equation*}
M(A|_U) = \rbmat{J_2(1)&J_2(1)\\0&J_2(1)} = \rbmat{1&1&1&1\\0&1&0&1\\0&0&1&1\\0&0&0&1},  \quad\text{so}\quad M(I+A|_U) = \rbmat{0&1&1&1\\0&0&0&1\\0&0&0&1\\0&0&0&0}.
\end{equation*}
Hence $\Ker(I+A|_U)=\Img(I+A|_U)=\hag{e_1,e_2+e_3}$.
\end{proof}

\begin{remark}\label{r:iso1} 
Denote by $A'$ a copy of $A$.
The {\it simplicial deleted join} $G^{*2}_\Delta$ of a complex $G=(V,F)$ is the complex with the set $V\sqcup V'$ of vertices and the set $\{\sigma*\tau'\ :\ \sigma,\tau\in F,\ \sigma\cap\tau=\varnothing\}$ of faces (here $\sigma$ and $\tau$ are considered as faces of a simplex, so that $\sigma*\tau'$ is the face spanned by all the  vertices of $\sigma$ and of $\tau'$; if $\sigma$ and $\tau$ are considered as subsets, then $\sigma*\tau'$ should be replaced by  $\sigma\sqcup\tau'$).

Analogously to the isomorphism $\varphi$ from the proof of Theorem \ref{p:symgrak}.a one proves that for complexes $A,B$ of dimensions $a,b$ the map $(A*B)^{\times2}_\Delta\to A^{*2}_\Delta\times B^{*2}_\Delta$ given by $(\alpha_1*\beta_1)\times(\alpha_2*\beta_2)\mapsto(\alpha_1*\alpha_2)\times(\beta_1*\beta_2)$ on $(a+b+2)$-cells induces an isomorphism of the groups $H_{a+b+2}$ of cellular $(a+b+2)$-cycles.
This is the analogue of \cite[Lemma 5.5.2]{Ma03} for deleted products. 

An alternative definition of the isomorphism $\varphi$ is the composition  
$$H_2((K_{n,n})^{\times2}_\Delta)\overset{(a)}\cong H_3((K_{n,n})^{*2}_\Delta)\overset{(b)}\cong H_3(\t{K_n}*\t{K_n})\overset{(c)}\cong H_2(\t{K_n}\times\t{K_n}),\quad\text{where}$$

$\bullet$ isomorphism (a) holds by \cite[Lemma 4d]{PS20}, cf. \cite[7.2.8c and 7.2.3b]{Sk}; it is induced by the map $\sigma\times\tau\mapsto\sigma*\tau$ from the set of $2$-cells to the set of  $3$-faces; 


$\bullet$ isomorphism (b) holds because $K_{n,n}=[n]^{*2}$ and $\t{K_n}=[n]^{*2}_\Delta$, so 
$(K_{n,n})^{*2}_\Delta\cong\t{K_n}*\t{K_n}$ \cite[Lemma 5.5.2]{Ma03}; the latter homeomorphism maps $(\alpha_1*\alpha_2)*(\beta_1*\beta_2)$ to $(\alpha_1*\beta_1)*(\alpha_2*\beta_2)$ for 
 $\alpha,\beta\in[n]^2$ such that $\alpha_j\ne\beta_j$ for every $j\in[2]$; 


$\bullet$ isomorphism (c) is induced by the map $\sigma*\tau\mapsto\sigma\times\tau$ from the set of $3$-faces to the set of  $2$-cells.  
\end{remark}

}

{\it In this list books, surveys, and expository papers are marked by stars}


\begin{thebibliography}{99}

\UseRawInputEncoding

\newcommand{\aate}{\bibitem[AA38]{AA38} \emph{A. Adrian Albert}. Symmetric and alternate matrices in an arbitrary field, I. Trans. Amer. Math. Soc., (1938) 43(3):386--436.}

\newcommand{\abc}{\bibitem[ABC+]{ABC+} * \emph{M. Atiyah, A. Borel, G. J. Chaitin, D. Friedan, J. Glimm, J. J. Gray, M. W. Hirsch, S. MacLane, B. B. Mandelbrot, D. Ruelle, A. Schwarz, K. Uhlenbeck, R. Thom, E. Witten, C.  Zeeman.} Responses to ``Theoretical Mathematics: Toward a cultural synthesis of mathematics and theoretical physics'', by A. Jaffe and F. Quinn. Bull. Am. Math. Soc. 30 (1994) 178--207. arXiv:math/9404229.}

\newcommand{\abgmns}{\bibitem[ABM+]{ABM+} * \emph{E. Alkin, E. Bordacheva, A. Miroshnikov, O. Nikitenko, A. Skopenkov,} Invariants of almost embeddings of graphs in the plane: results and problems, arXiv:2408.06392.}

\newcommand{\abms}{\bibitem[ABM+]{ABM+} * \emph{Э. Алкин, Е. Бордачева, А. Мирошников, А. Скопенков,} Инварианты почти вложений графов в плоскость, arXiv:2410.09860.}

\newcommand{\ams}{\bibitem[AMS]{AMS} * \emph{Э. Алкин, А. Мирошников, А. Скопенков,} Инварианты почти вложений графов в плоскость, arXiv:2410.09860v2.}

\newcommand{\adnt}{\bibitem[Ad93]{Ad93} * \emph{M. Adachi}. Embeddings and Immersions. Amer. Math.
Soc., 1993. (Transl. of Math. Monographs; V.~124).}

\newcommand{\adoe}{\bibitem[Ad18]{Ad18} {\it K. Adiprasito,} Combinatorial Lefschetz theorems beyond positivity, arXiv:1812.10454v4.}

\newcommand{\adnsv}{\bibitem[ADN+]{ADN+} * \emph{E. Alkin, S. Dzhenzher, O. Nikitenko, A. Skopenkov, A. Voropaev.}
Cycles in graphs and in hypergraphs: results and problems, arXiv:2308.05175.}

\newcommand{\agles}{\bibitem[AGL]{AGL86} Mathematical Economics,  ed. by A. Ambrosetti, F. Gori, R. Lucchetti,
Lect. Notes Math. 1330, Springer, 1986.}


\newcommand{\akzz}{\bibitem[Ak00]{Ak00} * \emph{П. М. Ахметьев.} Вложения компактов, стабильные
гомотопические группы сфер и теория особенностей, Успехи Мат. Наук.  2000. 55:3. C.~3-62.}

\newcommand{\akoe}{\bibitem[AK19]{AK19} \emph{S. Avvakumov, R. Karasev.} Envy-free division using mapping degree,
Mathematika, 67:1 (2020), 36--53. arXiv:1907.11183.}

\newcommand{\akto}{\bibitem[AK21]{AK21} \emph{G. Arone and V. Krushkal.}
Embedding obstructions in $\R^d$ from the Goodwillie-Weiss calculus and Whitney disks, 	Asian J. Math. 27 (2023), 135--186. arXiv:2101.10995. }

\newcommand{\akm}{\bibitem[AKM]{AKM} \emph{M. Abrahamsen, L. Kleist and T. Miltzow.}
Geometric Embeddability of Complexes is $\exists\mathbb R$-complete. arXiv:2108.02585.}

\newcommand{\aksoe}{\bibitem[AKS]{AKS} \emph{S. Avvakumov, R. Karasev and A. Skopenkov.} Stronger counterexamples to the topological Tverberg conjecture, Combinatorica, 43 (2023), 717--727. arXiv:1908.08731.}


\newcommand{\akuoe}{\bibitem[AKu19]{AKu19} \emph{S. Avvakumov, S. Kudrya.}
Vanishing of all equivariant obstructions and the mapping degree.
Discr. Comp. Geom., 66:3 (2021) 1202--1216. arXiv:1910.12628.}

\newcommand{\aktf}{\bibitem[Ak]{Ak} \emph{Д. Акимов,} Нечетность суммы чисел оборотов для почти вложения графа $K_{3,2}$.}

\newcommand{\alto}{\bibitem[Al22]{Al22} \emph{E. Alkin,}
Hardness of almost embedding simplicial complexes in $\R^d$, II. arXiv:2206.13486}

\newcommand{\amtf}{\bibitem[AM25]{AM25} \emph{E. Alkin, A. Miroshnikov,} On winding numbers of almost embeddings of $K_4$ in the plane, Math. Notes, to appear, arXiv:2501.15642.}

\newcommand{\amsw}{\bibitem[AMS+]{AMSW} \emph{S. Avvakumov, I. Mabillard, A. Skopenkov and U. Wagner.}
Eliminating Higher-Multiplicity Intersections, III. Codimension 2, Israel J. Math. 245 (2021) 501--534.  arxiv:1511.03501.}


\newcommand{\anzt}{\bibitem[An03]{An03} * \emph{Д. В. Аносов.} Отображения окружности, векторные поля и их применения. М: МЦНМО, 2003.}

\newcommand{\anstf}{\bibitem[ANS]{ANS} \emph{E. Alkin, O. Nikitenko, A. Skopenkov,} Homotopy classification of closed polygonal lines: results and problems, arXiv:2508.16287.}

\newcommand{\arnf}{\bibitem[Ar95]{Ar95} * \emph{V. I. Arnold,}  Topological invariants of plane curves and caustics, University Lecture Series, Vol. 5, Amer. Math. Soc., Providence, RI, 1995.}

\newcommand{\arszo}{\bibitem[ARS01]{ARS01} \emph{P. Akhmetiev, D. Repov\v s and A. Skopenkov},
Embedding products of low-dimensional manifolds in $\R^m$, Topol. Appl. 113 (2001), 7--12.}

\newcommand{\arszt}{\bibitem[ARS02]{ARS02} \emph{P. Akhmetiev, D. Repovs and A. Skopenkov.} Obstructions to approximating maps of $n$-manifolds into $R^{2n}$ by embeddings, Topol. Appl., 123 (2002), 3--14.}

\newcommand{\asts}{\bibitem[AS26]{AS26} \emph{E. Alkin and A. Skopenkov,} Homotopy classification of closed polygonal lines.}

\newcommand{\asoed}{\bibitem[As]{As} \emph{A. Asanau,} \lowercase{A SIMPLE PROOF THAT CONNECTED SUM OF ORDERED
ORIENTED LINKS IS NOT WELL-DEFINED,} Math. Notes, to appear.}

\newcommand{\asoe}{\bibitem[As]{As} \emph{A. Asanau,} On the \lowercase{TRIPLE SELF-INTERSECTION NUMBER FOR GRAPHS IN THE PLANE,} unpublished, 2018.}

\newcommand{\avos}{\bibitem[Av14]{Av14} \emph{S. Avvakumov,} The classification of certain linked 3-manifolds in 6-space, Moscow Math. J., 16:1 (2016), 1--25. arXiv:1408.3918.}

\newcommand{\avose}{\bibitem[Av17]{Av17} \emph{S. Avvakumov,} The classification of linked 3-manifolds in 6-space, Algebraic \& Geometric Topology, 22:6 (2022) 2587--2630. arXiv:1704.06501.}



\newcommand{\bant}{\bibitem[Ba93]{Ba93} * \emph{T. Bartsch.} Topological methods for variational problems with
symmetries, Lecture Notes in Mathematics, 1560, Springer-Verlag, Berlin, 1993.}

\newcommand{\batt}{\bibitem[Ba23]{Ba23} * \emph{I. Barany.} Tverberg's theorem, a new proof. arXiv:2308.10105.}

\newcommand{\bbsn}{\bibitem[BB79]{BB} \emph{E.~G. Bajm{{\'o}}czy and I.~B{{\'a}}r{{\'a}}ny,}
\newblock On a common generalization of {B}orsuk's and {R}adon's theorem,
\newblock Acta Math.\ Acad.\ Sci.\ Hungar.\ 34:3 (1979), 347-350.}

\newcommand{\bbzos}{\bibitem[BBZ]{BBZ} * \emph{I.~B{{\'a}}r{{\'a}}ny, P.~V.~M. Blagojevi{{\'c}} and G.~M. Ziegler.} Tverberg's Theorem at 50: Extensions and Counterexamples, Notices of the Amer. Math. Soc., 63:7 (2016), 732--739.}


\newcommand{\bcm}{\bibitem[BCM]{BCM} * 13th Hilbert Problem on superpositions of functions, presented by A. Belov, A. Chilikov, I. Mitrofanov, S. Shaposhnikov and A. Skopenkov,
\url{http://www.turgor.ru/lktg/2016/5/index.htm}.}

\newcommand{\beet}{\bibitem[BE82]{BE82} * \emph{V.G. Boltyansky and V.A. Efremovich.} Intuitive Combinatorial Topology. Springer.}

\newcommand{\beetr}{\bibitem[BE82]{BE82} * \emph{В. Г. Болтянский и В. А. Ефремович.} Наглядная топология. М.:  Наука, 1982.}


\newcommand{\bfzn}{\bibitem[BF09]{BF09} \emph{K. Barnett, M. Farber}. Topology of Configuration Space of Two Particles on a Graph, I.  Algebr. Geom. Topol. 9 (2009) 593--624.	arXiv:0903.2180.}

\newcommand{\bfzof}{\bibitem[BFZ14]{BFZ14} \emph{P. V. M. Blagojevi{\'c}, F. Frick, and G. M. Ziegler,}
Tverberg plus constraints, Bull. Lond. Math. Soc. 46:5 (2014), 953-967, arXiv:1401.0690.}


\newcommand{\bfzos}{\bibitem[BFZ]{BFZ} \emph{P. V. M. Blagojevi{\'c}, F. Frick and G. M. Ziegler,}
Barycenters of Polytope Skeleta and Counterexamples to the Topological Tverberg Conjecture, via Constraints,
J. Eur. Math. Soc., 21:7 (2019) 2107-2116. arXiv:1510.07984.}


\newcommand{\bgso}{\bibitem[BG71]{BG71} J.C. Becker and H. H. Glover, {\it Note on the Embedding of Manifolds in Euclidean Space,} Proc. of the Amer. Math. Soc., 27:2 (1971) 405-410.}


\newcommand{\bgos}{\bibitem[BG16]{BG16} \emph{A. Bj\"orner and A. Goodarzi}, On Codimension one Embedding of Simplicial Complexes, in book: A Journey Through Discrete Mathematics, arXiv:1605.01240.}

\newcommand{\biet}{\bibitem[Bi83]{Bi83} * \emph{R. H. Bing.} The Geometric Topology of 3-Manifolds. Providence, R.~I. 1983. (Amer. Math. Soc. Colloq. Publ., 40).}

\newcommand{\bitz}{\bibitem[Bi20]{Bi20} * \emph{A. Bikeev.} Realizability of discs with ribbons on the M\"obius strip. Mat. Prosveschenie, 28 (2021), 150-158;
erratum to appear. arXiv:2010.15833.}

\newcommand{\bitzr}{\bibitem[Bi20]{Bi20} * \emph{А. Бикеев.} Реализуемость дисков с ленточками на ленте Мебиуса.
Мат. просвещение. Сер. 3. 28 (2021), 150--158.}

\newcommand{\bito}{\bibitem[Bi21]{Bi21} {\it A. I. Bikeev,}
Criteria for integer and modulo 2 embeddability of graphs to surfaces, arXiv:2012.12070v2.}


\newcommand{\bagos}{\bibitem[BG17]{BG17} \emph{S. Basu and S. Ghosh.} Equivariant maps related to the topological Tverberg conjecture, Homology, Homotopy and Applications 19:1 (2017) 155--170.}

\newcommand{\bkkmzof}{\bibitem[BKK]{BKK} \emph{M. Bestvina, M. Kapovich and B. Kleiner,}
Van Kampen's embedding obstruction for discrete groups, Invent. Math. 150 (2002) 219--235. arXiv:math/0010141.}

\newcommand{\bl}{\bibitem[BL]{BL} \url{https://en.wikipedia.org/wiki/Brunnian_link}}

\newcommand{\blf}{\bibitem[BL4]{BL4} Students form a 4-component Brunnian link,  \url{http://www.mccme.ru/circles/oim/foto2014/brunn4.png} (5Mb)}

\newcommand{\bmzf}{\bibitem[BM04]{BM04} \emph{Boyer, J. M. and Myrvold, W. J.} On the cutting edge: simplified $O(n)$ planarity by edge addition,  Journal of Graph Algorithms and Applications, 8:3 (2004) 241--273.}

\newcommand{\bm}{\bibitem[BM15]{BM15} \emph{I. Bogdanov and A. Matushkin.} Algebraic proofs of linear versions of the Conway--Gordon--Sachs theorem and the van Kampen--Flores theorem, arXiv:1508.03185.}


\newcommand{\bmzzn}{\bibitem[BMZ09]{BMZ09} \emph{P. V. M. Blagojevi{\'c}, B. Matschke, G. M. Ziegler,}
Optimal bounds for a colorful Tverberg-Vre\'cica type problem, Advances in Math., 226 (2011), 5198-5215, arXiv:0911.2692.}

\newcommand{\bmzof}{\bibitem[BMZ15]{BMZ15} \emph{P. V. M. Blagojevi{\'c}, B. Matschke, G. M. Ziegler,}
Optimal bounds for the colored Tverberg problem, J. Eur. Math. Soc.,  17:4 (2015) 739--754,
arXiv:0910.4987.}

\newcommand{\bpns}{\bibitem[BP97]{BP97} * \emph{R. Benedetti and C. Petronio.} Branched standard spines of 3-manifolds, Lecture Notes in Math. 1653, Springer-Verlag, Berlin-Heidelberg-New York, 1997.}

\newcommand{\brst}{\bibitem[Br72]{Br72} \emph{J. L. Bryant.} Approximating embeddings of polyhedra in codimension 3, Trans. Amer. Math. Soc., 170 (1972) 85--95.}

\newcommand{\brts}{\bibitem[Br68]{Br68} \emph{P. Bruegel,} 1568,
\url{https://en.wikipedia.org/wiki/The_Magpie_on_the_Gallows}.}


\newcommand{\bren}{\bibitem[Br82]{brown1982} * \emph{K.~S. Brown.} \newblock Cohomology of Groups. \newblock Springer-Verlag New York, 1982.}


\newcommand{\bssos}{\bibitem[BS17]{BS17} * \emph{I.~B\'{a}r\'{a}ny and P. Sober\'{o}n,} Tverberg's theorem is 50 years old: a survey, Bull. Amer. Math. Soc. (N.S.) 55:4 (2018), 459--492. arXiv:1712.06119.}

\newcommand{\bsto}{\bibitem[BS21]{BS21} * \emph{A. Buchaev and A. Skopenkov,} Simple proofs of estimations of Ramsey numbers and of discrepancy, Mat. Prosveschenie, to appear, arXiv:2107.13831.}

\newcommand{\brsnn}{\bibitem[BRS99]{BRS99} \emph{D. Repov\v s, N. Brodsky and A. B. Skopenkov.}
A classification of 3-thickenings of 2-polyhedra, Topol. Appl. 1999. 94. P.~307-314.}

\newcommand{\bsseo}{\bibitem[BSS]{BSS} \emph{I.~B\'{a}r\'{a}ny, S.~B. Shlosman, and A.~Sz{\H{u}}cs,}
\newblock On a topological generalization of a theorem of {T}verberg,
\newblock J.\ London Math.\ Soc.\ (II. Ser.) 23 (1981), 158--164.}

\newcommand{\btzs}{\bibitem[BT07]{BT07} \emph{A. Bj\"orner, M. Tancer}, Combinatorial Alexander Duality --- a Short and Elementary Proof, Discr. and Comp. Geom., 42 (2009) 586. arXiv:0710.1172.}

\newcommand{\buse}{\bibitem[Bu68]{Bu68} \emph{A. R. Butz,} Space filling curves and mathematical programming, Information and Control, 12:4 (1968) 314--330.}


\newcommand{\bz}{\bibitem[BZ16]{BZ16} * \emph{P. V. M. Blagojevi\'c and G. M. Ziegler,} Beyond the Borsuk-Ulam theorem: The topological Tverberg story, in: A Journey Through Discrete Mathematics, Eds. M. Loebl,
J. Ne\v set\v ril, R. Thomas, Springer, 2017, 273--341. arXiv:1605.07321v3.}



\newcommand{\cano}{\bibitem[Ca91]{Ca91} * \emph{D. de Caen}, The ranks of tournament matrices, Amer. Math. Monthly, 98:9 (1991) 829--831.}

\newcommand{\ca}{\bibitem[Ca]{Ca} \emph{J. Carmesin.} Embedding simply connected 2-complexes in 3-space, I-V, arXiv:1709.04642, arXiv:1709.04643, arXiv:1709.04645, arXiv:1709.04652, arXiv:1709.04659.}

\newcommand{\cfsz}{\bibitem[CF60]{CF60} \emph{P. E. Conner and E. E. Floyd}, Fixed points free involutions and equivariant maps, Bull. Amer. Math. Soc., 66 (1960) 416--441.}

\newcommand{\cfs}{\bibitem[CFS]{CFS} \emph{D. Crowley, S.C. Ferry, M. Skopenkov,} The rational classification of links of codimension $>2$, Forum Math. 26 (2014), 239--269. arXiv:1106.1455.}

\newcommand{\cget}{\bibitem[CG83]{CG83} \emph{J. H. Conway and C. M. A. Gordon},
Knots and links in spatial graphs, J. Graph Theory  7 (1983), 445--453.}

\newcommand{\cten}{\bibitem[Ch]{Ch} \emph{Chuang Tzu,} translated by H. A. Giles, Bernard Quaritch, London, 1889.}

\newcommand{\ctruku}{\bibitem[Ch]{Ch} \emph{Chuang Tzu,} translated to Russian by S. Kuchera, in: Ancient Chinese Philosophy, v. I, Mysl, Moscow, 1972.}


\newcommand{\chnn}{\bibitem[Ch99]{Ch99} * \emph{А. В. Чернавский,} Теорема Жордана.  Мат. Просвещение, 3 (1999), 142--157.}

\newcommand{\hcon}{\bibitem[HC19]{HC19} * \emph{C. Herbert Clemens.} Two-Dimensional Geometries. A Problem-Solving Approach, Amer. Math. Soc., 2019.}

\newcommand{\ckmoo}{\bibitem[CKMS]{CKMS} \emph{M. \v Cadek, M. Kr\v c\'al. J. Matou\v sek, F. Sergeraert,
L. Vok\v r\'inek, U. Wagner.} Computing all maps into a sphere, J. of the ACM, 61:3 (2014). arXiv:1105.6257.}


\newcommand{\ckmvwot}{\bibitem[CKM12+]{CKM12+} \emph{M. \v Cadek, M. Kr\v c\'al. J. Matou\v sek, L. Vok\v r\'inek, U. Wagner.} Polynomial-time computation of homotopy groups and Postnikov systems in fixed dimension, SIAM J. Comput., 43:5 (2014), 1728--1780. arXiv:1211.3093.}

\newcommand{\ckmvw}{\bibitem[CKM+]{CKM+} \emph{M. \v Cadek, M. Kr\v c\'al. J. Matou\v sek, L. Vok\v r\'inek, U. Wagner.} Extendability of continuous maps is undecidable, Discr. and Comp. Geom. 51 (2014) 24--66.
arXiv:1302.2370.}

\newcommand{\ckppt}{\bibitem[CKP+]{CKP+} \emph{E. Colin de Verdi\'ere, V. Kalu\v za, P. Pat\'ak, Z. Pat\'akov\'a and M. Tancer.} A direct proof of the strong Hanani-Tutte theorem on the projective plane. Journal of Graph Algorithms and Applications, 21:5 (2017) 939--981.}

\newcommand{\cksof}{\bibitem[CKS+]{CKS+} * New ways of weaving baskets, presented by G. Chelnokov, Yu. Kudryashov, A.Skopenkov and A. Sossinsky, \url{http://www.turgor.ru/lktg/2004/lines.en/index.htm}.}

\newcommand{\ckv}{\bibitem[CKV]{CKV} \emph{M.~{\v{C}}adek, M.~Kr\v{c}\'{a}l, and L.~Vok\v{r}\'{\i}nek.}
Algorithmic solvability of the lifting-extension problem, Discr. Comp. Geom. 57 (2017), 915--965. arXiv:1307.6444.}


\newcommand{\clr}{\bibitem[CLR]{CLR} * \emph{Т. Кормен, Ч. Лейзерсон, Р. Ривест.} Алгоритмы:
построение и анализ, МЦНМО, Москва, 1999.}

\newcommand{\clreng}{\bibitem[CLR]{CLR} * \emph{T. H. Cormen, C. E.Leiserson, R. L.Rivest, C. Stein.} Introduction to Algorithms, MIT Press, 2009.}

\newcommand{\crzfru}{\bibitem[CR]{CR} * \emph{Р. Курант, Дж. Роббинс,} Что такое математика. М.: МЦНМО, 2004.}

\newcommand{\crzfen}{\bibitem[CR]{CR} * \emph{R. Courant and H. Robbins,} What is Mathematics, Oxford Univ. Press.}

\newcommand{\crsne}{\bibitem[CRS98]{CRS98} * \emph{A. Cavicchioli, D. Repov\v s and A. B. Skopenkov.}
Open problems on graphs, arising from geometric topology, Topol. Appl. 84 (1998), 207--226.}

\newcommand{\crsos}{\bibitem[CRS']{CRS'} \emph{M. Cencelj, D. Repov\v s and M. Skopenkov,} Homotopy type of the complement of an immersion and classification of embeddings of tori, Russian Math. Surv. 62:5 (2007), 985--987,
arXiv:0803.4285.}

\newcommand{\crsot}{\bibitem[CRS]{CRS} \emph{M. Cencelj, D. Repov\v s and M. Skopenkov,}
Classification of knotted tori in the 2-metastable dimension, Mat. Sbornik, 203:11 (2012), 1654--1681.
arxiv:0811.2745.}

\newcommand{\csoo}{\bibitem[CS08]{CS08} \emph{D. Crowley and A. Skopenkov.} A classification of smooth embeddings of 4-manifolds in 7-space, II, Intern. J. Math., 22:6 (2011) 731-757, arxiv:0808.1795.}

\newcommand{\csos}{\bibitem[CS16]{CS16} \emph{D. Crowley and A. Skopenkov,} Embeddings of non-simply-connected 4-manifolds in 7-space. I. Classification modulo knots, Moscow Math. J., 21 (2021), 43--98. arXiv:1611.04738.}


\newcommand{\csoso}{\bibitem[CS16o]{CS16o} \emph{D. Crowley and A. Skopenkov,} Embeddings of non-simply-connected 4-manifolds in 7-space. II. On the smooth classification, Proc. A of the Royal Soc. of Edinburgh 152:1 (2022), 163--181. arXiv:1612.04776.}


\newcommand{\crsk}{\bibitem[CS]{CS} \emph{D. Crowley and A. Skopenkov,} Embeddings of non-simply-connected 4-manifolds in 7-space. III. Piecewise-linear classification. draft.}

\newcommand{\cutz}{\bibitem[Cu20]{Cu20} \emph{C. Culter,} Cantor sets are not tangent homogeneous,
Topol. Appl. 271 (2020) 1--9.}


\newcommand{\dies}{\bibitem[Di87]{Di} * \emph{T. tom Dieck,} Transformation groups, Studies in Mathematics, vol. 8, Walter de Gruyter, Berlin, 1987.}

\newcommand{\dize}{\bibitem[Di08]{Di08} * \emph{T. tom Dieck,} Algebraic topology, EMS Textbooks in Mathematics, 
EMS, Z\"urich, 2008.}

\newcommand{\dent}{\bibitem[De93]{De93}  \emph{T.K. Dey.} On counting triangulations in $d$-dimensions. Comput. Geom.  3:6 (1993) 315--325.}

\newcommand{\denf}{\bibitem[DE94]{DE94}  \emph{T.K. Dey and H. Edelsbrunner.} Counting triangle crossings and halving planes, Discrete Comput. Geom. 12 (1994), 281--289.}

\newcommand{\dgn}{\bibitem[DGN+]{DGN+} * S. Dzhenzher, T. Garaev, O. Nikitenko, A. Petukhov, A. Skopenkov, A. Voropaev, Low rank matrix completion and realization of graphs: results and problems, arXiv:2501.13935.}

\newcommand{\dgnr}{\bibitem[DGN+]{DGN+} * Минимизация ранга восполнением матриц, представляли А. Воропаев, Т. Гараев, С. Дженжер, О. Никитенко, А. Петухов и А. Скопенков, \url{https://www.mccme.ru/circles/oim/netflix_rus.pdf}.}

\newcommand{\dstt}{\bibitem[DS22]{DS22}  \emph{S. Dzhenzher and A. Skopenkov,} A quadratic estimation for the K\"uhnel conjecture on embeddings, arXiv:2208.04188.}

\newcommand{\botf}{\bibitem[Dz25]{Dz25} \emph{E. Dzhenzher,} Symmetric 1-cycles in the deleted product of a graph, Topol. Appl. (2025) 109277.}



\newcommand{\embo}{\bibitem[Eb]{Eb} * \url{http://www.map.mpim-bonn.mpg.de/Embeddings_of_manifolds_with_boundary:_classification}}

\newcommand{\embe}{\bibitem[Em]{Em} * \url{http://www.map.mpim-bonn.mpg.de/Embedding_(simple_definition)}}

\newcommand{\ers}{\bibitem[ERS]{ERS} * Invariants of graph drawings in the plane, presented by A. Enne, A. Ryabichev, A. Skopenkov and T. Zaitsev, \url{http://www.turgor.ru/lktg/2017/6/index.htm}}



\newcommand{\feto}{\bibitem[Fe21]{Fe21} \emph{M. Fedorov.} A description of values of Seifert form for punctured $n$-manifolds in $(2n-1)$-space, arXiv:2107.02541.}

\newcommand{\ffen}{\bibitem[FF89]{FF89} * \emph{А. Т. Фоменко и Д. Б. Фукс.} Курс гомотопической топологии. М.: Наука, 1989.}

\newcommand{\ffene}{\bibitem[FF89]{FF89} * \emph{A.T. Fomenko and D.B. Fuchs.} Homotopical Topology, Springer, 2016.}


\newcommand{\fhzo}{\bibitem[FH10]{FH10}  \emph{M. Farber, E. Hanbury}. Topology of Configuration Space of Two Particles on a Graph, II. Algebr. Geom. Topol. 10 (2010) 2203--2227. arXiv:1005.2300.}


\newcommand{\fkosc}{\bibitem[FK17]{FK17} \emph{R. Fulek, J. Kyn{\v{c}}l,} Counterexample to an Extension of the Hanani-Tutte Theorem on the Surface of Genus 4, Combinatorica, 39 (2019) 1267--1279, arXiv:1709.00508.}

\newcommand{\fkos}{\bibitem[FK17]{FK17} \emph{R. Fulek, J. Kyn{\v{c}}l,} Hanani-Tutte for approximating maps of graphs, arXiv:1705.05243.}

\newcommand{\fkon}{\bibitem[FK19]{FK19} \emph{R. Fulek, J. Kyn{\v{c}}l,}
$\Z_2$-genus of graphs and minimum rank of partial symmetric matrices,
35th Intern. Symp. on Comp. Geom. (SoCG 2019), Article No. 39; pp. 39:1--39:16, \linebreak
\url{https://drops.dagstuhl.de/opus/volltexte/2019/10443/pdf/LIPIcs-SoCG-2019-39.pdf}.
We refer to numbering in arXiv version: arXiv:1903.08637.}

\newcommand{\fktnf}{\bibitem[FKT]{FKT} \emph{M. H. Freedman, V. S. Krushkal and P. Teichner.} Van Kampen's
embedding obstruction is incomplete for 2-complexes in~$\R^4$, Math. Res. Letters. 1994. 1. P.~167-176.}

\newcommand{\fltf}{\bibitem[Fl34]{Fl34} \emph{A. Flores}, \"Uber $n$-dimensionale Komplexe die im $E^{2n+1}$ absolut selbstverschlungen sind, Ergeb. Math. Koll. 6 (1934) 4--7.}

\newcommand{\fnzn}{\bibitem[FN09]{FN09} \emph{T. Fleming and R. Nikkuni,} Homotopy on spatial graphs and the Sato-Levine invariant, Trans. Amer. Math. Soc. 361:4 (2009), 1885--1902.}

\newcommand{\fo}{\bibitem[Fo]{Fo} * \emph{L. Fortnow.} Time for Computer Science to Grow Up,  \url{https://people.cs.uchicago.edu/~fortnow/papers/growup.pdf}.}

\newcommand{\fozf}{\bibitem[Fo04]{Fo04} * \emph{R. Fokkink.} A forgotten mathematician, Eur. Math. Soc. Newsletter 52 (2004) 9--14.}


\newcommand{\fpstz}{\bibitem[FPS]{FPS} \emph{R. Fulek, M.J. Pelsmajer and M. Schaefer.}
Strong Hanani-Tutte for the Torus, arXiv:2009.01683.}

\newcommand{\frse}{\bibitem[Fr78]{Fr78} \emph{M. Freedman,} Quadruple points of 3-manifolds in $S^4$, Comment. Math. Helv. 53 (1978), 385-394.}

\newcommand{\fres}{\bibitem[FR86]{FR86} \emph{R. Fenn, D. Rolfsen.}
Spheres may link homotopically in 4-space, J. London Math. Soc. 34 (1986) 177-184.}

\newcommand{\frofea}{\bibitem[Fr15']{Fr15'} \emph{F. Frick}, Counterexamples to the topological Tverberg conjecture, arXiv:1502.00947v1.}


\newcommand{\frof}{\bibitem[Fr15]{Fr15} \emph{F. Frick}, Counterexamples to the topological Tverberg conjecture,
Oberwolfach reports, 12:1 (2015), 318--321. arXiv:1502.00947.}

\newcommand{\fros}{\bibitem[Fr17]{Fr17} \emph{F. Frick}, O\lowercase{N AFFINE TVERBERG-TYPE RESULTS WITHOUT CONTINUOUS GENERALIZATION}, arXiv:1702.05466}


\newcommand{\fstz}{\bibitem[FS20]{FS20} \emph{F. Frick and P. Sober\'on}, The topological Tverberg problem beyond prime powers, arXiv:2005.05251.}

\newcommand{\ftss}{\bibitem[FT77]{FT77} \emph{R. Fenn, P. Taylor,} Introducing doodles, pp. 37-43
in: Topology of Low-Dimensional Manifolds, Proceedings of the Second Sussex Conference, 1977,
Ed. R. Fenn, V. 722 of Lecture Notes in Math.}

\newcommand{\fvto}{\bibitem[FV21]{FV21} \emph{M. Filakovsk\'y, L. Vok\v r\'inek.} Computing homotopy classes for diagrams, Discr. Comp. Geom. 70 (2023), 866--920. arXiv:2104.10152.}

\newcommand{\fwz}{\bibitem[FWZ]{FWZ} \emph{M. Filakovsk\'y, U. Wagner, S. Zhechev.} Embeddability of simplicial complexes is undecidable. Oberwolfach reports, to appear.}

\newcommand{\fwztz}{\bibitem[FWZ]{FWZ} \emph{M. Filakovsk\'y, U. Wagner, S. Zhechev.} Embeddability of simplicial complexes is undecidable. Proceedings of the 2020 ACM-SIAM Symposium on Discrete Algorithms.}



\newcommand{\ga}{\bibitem[GA]{GA} * \url{https://en.wikipedia.org/wiki/Galactic_algorithm}}

\newcommand{\gatt}{\bibitem[Ga23]{Ga23} \emph{T. Garaev}, On drawing $K_5$ minus an edge in the plane, arXiv:2303.14503.}

\newcommand{\gdikrse}{\bibitem[GDI]{GDI} * {\it A. Chernov, A. Daynyak, A. Glibichuk, M. Ilyinskiy, A. Kupavskiy, A. Raigorodskiy and A. Skopenkov,} Elements of Discrete Mathematics As a Sequence of Problems (in Russian),
MCCME, Moscow, 2016. Update of a part: \url{http://www.mccme.ru/circles/oim/discrbook.pdf}}

\newcommand{\gdikrs}{\bibitem[GDI]{GDI} * {\it А.А. Глибичук, А.Б. Дайняк, Д.Г. Ильинский, А.Б. Купавский, А.М. Райгородский, А.Б. Скопенков, А.А. Чернов,} Элементы дискретной математики в задачах, М, МЦНМО, 2016.
Обновляемая версия части книги: \url{http://www.mccme.ru/circles/oim/discrbook.pdf}}

\newcommand{\giso}{\bibitem[Gi71]{Gi71} * {\it S. Gitler,} Immersion and Embedding of Manifolds,
Proc. Symp. Pure Math. 22, 87-96 (1971).}

\newcommand{\gkp}{\bibitem[GKP]{GKP} * {\it R. Graham, D. Knuth, and O. Patashnik,} Concrete Mathematics: A Foundation for Computer Science, Addison–Wesley, first published in 1989, \url{https://www.csie.ntu.edu.tw/~r97002/temp/Concrete\%20Mathematics\%202e.pdf}.}

\newcommand{\gmpptw}{\bibitem[GMP+]{GMP+} \emph{X. Goaoc, I. Mabillard, P. Pat\'ak, Z. Pat\'akov\'a, M. Tancer, U. Wagner}, On Generalized Heawood Inequalities for Manifolds: a van Kampen--Flores-type Nonembeddability Result,
Israel J. Math., 222(2) (2017) 841-866. arXiv:1610.09063.}


\newcommand{\gppot}{\bibitem[GPP+]{GPP+} \emph{X. Goaoc, P. Pat\'ak, Z. Pat\'akov\'a, M. Tancer, and U. Wagner.} Bounding Helly numbers via Betti numbers. In 31st International Symposium on Computational Geometry, volume 34
of LIPIcs. Leibniz Int. Proc. Inform., pp. 507-521. Schloss Dagstuhl. Leibniz-Zent. Inform., Wadern, 2015. Full version: arXiv:1310.4613.}

\newcommand{\group}{\bibitem[Gr]{Gr} * \url{https://en.wikipedia.org/wiki/Groupthink}}

\newcommand{\grsz}{\bibitem[Gr69]{Gr69} \emph{B. Gr\"unbaum.} Imbeddings of simplicial complexes. Comment. Math. Helv., 44:1, 502--513, 1969.}


\newcommand{\gres}{\bibitem[Gr86]{Gr86} * \emph{M. Gromov}, Partial Differential Relations,
Ergebnisse der Mathematik und ihrer Grenzgebiete (3), Springer Verlag, Berlin-New York, 1986.}

\newcommand{\groz}{\bibitem[Gr10]{Gr10} \emph{M. Gromov,}
\newblock Singularities, expanders and topology of maps. Part 2: From combinatorics to topology via algebraic isoperimetry, \newblock Geometric and Functional Analysis 20 (2010), no.~2, 416--526.}

\newcommand{\grsn}{\bibitem[GR79]{GR79} \emph{J. L. Gross	and R. H. Rosen}, A linear time planarity algorithm for 2-complexes, Journal of the ACM, 26:4 (1979), 611--617.}

\newcommand{\gs}{\bibitem[GS]{GS} \emph{М. Гортинский и О. Скрябин.} Критерий вложимости графов в плоскость вдоль прямой, препринт.}

\newcommand{\gssn}{\bibitem[GS79]{GS} \emph{P.~M. Gruber and R.~Schneider,} Problems in geometric convexity. In {\em Contributions to geometry (Proc. Geom. Sympos., Siegen, 1978)}, 255--278. Birkh{\"a}user, Basel-Boston, Mass., 1979.}

\newcommand{\gsnn}{\bibitem[GS99]{GS99} \emph{R. Gompf and A. Stipsicz,}
4-manifolds and Kirby calculus, GSM20, AMS, Providence, RI, 1999.}


\newcommand{\gszs}{\bibitem[GS06]{GS06} \emph{D. Goncalves and A. Skopenkov,} Embeddings of homology equivalent manifolds with boundary, Topol. Appl., 153:12 (2006) 2026-2034. arxiv:1207.1326.}

\newcommand{\gssoe}{\bibitem[GSS+]{GSS+} * Projections of skew lines, presented by A. Gaifullin, A. Shapovalov, A. Skopenkov and M. Skopenkov, \url{http://www.turgor.ru/lktg/2001/index.php}.}

\newcommand{\gtes}{\bibitem[GT87]{GT87} * \emph{J. L. Gross and T. W. Tucker.}
Topological graph theory. New York: Wiley-Interscience, 1987.}

\newcommand{\guzn}{\bibitem[Gu09]{Gu09} \emph{A. Gundert.} On the complexity of embeddable simplicial complexes. Diplomarbeit, Freie Universit\"at Berlin, 2009. 	arXiv:1812.08447.}


\newcommand{\ha}{\bibitem[Ha]{Ha} * \emph{F. Harary.} Graph theory.
Рус. пер.: Ф. Харари. Теория графов. М., Мир, 1973.}

\newcommand{\hats}{\bibitem[Ha37]{Ha37} \emph{W. Hantzsche,} Einlagerung von Mannigfaltigkeiten in euklidische R\" aume, Math. Zeitschrift, 43:1 (1937) 38--58.}

\newcommand{\hastk}{\bibitem[Ha62k]{Ha62k} {\em A.~Haefliger,}  Knotted $(4k-1)$-spheres in $6k$-space, Ann. of Math. 75 (1962) 452--466.}

\newcommand{\hastl}{\bibitem[Ha62l]{Ha62l} \emph{A. Haefliger,} Differentiable links, Topology, 1 (1962) 241--244.}

\newcommand{\hast}{\bibitem[Ha63]{Ha63} \emph{A.~Haefliger,} Plongements differentiables dans le domain stable, Comment. Math. Helv. 36 (1962-63) 155--176.}

\newcommand{\hassa}{\bibitem[Ha66A]{Ha66A} \textit{A. Haefliger}. Differential embeddings of~$S^n$ in $S^{n+q}$ for $q>2$. Ann. Math. (2), 83 (1966), 402--~436.}

\newcommand{\hass}{\bibitem[Ha66C]{Ha66C} \emph{A.~Haefliger,}  Enlacements de spheres en codimension superiure \`a 2, Comment. Math. Helv. 41 (1966-67) 51--72.}

\newcommand{\hase}{\bibitem[Ha68]{Ha68} \emph{A. Haefliger,} Knotted Spheres and Related Geometric Topic,
in Proc. Int. Congr. Math., Moscow, 1966 (Mir, Moscow, 1968), 437--445.}

\newcommand{\hasn}{\bibitem[Ha69]{Ha69} \emph{L.~S.~Harris,} Intersections and embeddings of polyhedra, Topology 8 (1969) 1--26.}

\newcommand{\hasf}{\bibitem[Ha74]{Ha74} * \emph{P. Halmos,} How to talk mathematics. Notices of the Amer. Math. Soc., 21 (1974) 155--158.}

\newcommand{\haef}{\bibitem[Ha84]{Ha84} \emph{N. Habegger,} Obstruction to embedding disks II: a proof of a conjecture by Hudson, Topol. Appl. 17 (1984).}

\newcommand{\haes}{\bibitem[Ha86]{Ha86} \emph{N. Habegger,} Knots and links in codimension greater than 2, Topology, 25:3 (1986) 253--260.}

\newcommand{\hogr}{\bibitem[HG]{HG} * \url{http://www.map.mpim-bonn.mpg.de/Homology_groups_(simplicial;_simple_definition)}}

\newcommand{\hifn}{\bibitem[Hi59]{Hi59} \emph{M. W. Hirsch.} Immersions of manifolds, Trans. Amer. Math. Soc. 93 (1959) 242--276.}

\newcommand{\hjsf}{\bibitem[HJ64]{HJ64} \emph{R. Halin and H. A. Jung.}
Karakterisierung der Komplexe der Ebene und der 2-Sph\"are, Arch. Math. 1964. 15. P.~466-469.}

\newcommand{\hkns}{\bibitem[HK96]{HK96} \emph{N. Habegger and U. Kaiser,} Homotopy classes of 2 disjoint $2p$-spheres in $\R^{3p+1}$, Topol. Appl. 71 (1996) 1--8.}

\newcommand{\hkne}{\bibitem[HK98]{HK98} \emph{N. Habegger and U. Kaiser,} Link homotopy in 2--metastable range, Topology 37:1 (1998) 75--94.}

\newcommand{\hmsnt}{\bibitem[HMS]{HMS93} * \emph{C. Hog-Angeloni, W. Metzler and A. J. Sieradski.}
Two-dimensional homotopy and combinatorial group theory. Cambridge: Cambridge Univ. Press, 1993. (London Math. Soc. Lecture Notes, 197).}

\newcommand{\ho}{\bibitem[Ho]{Ho} * The Hopf fibration, \url{https://www.youtube.com/watch?v=AKotMPGFJYk}}

\newcommand{\hozs}{\bibitem[Ho06]{Ho06} \emph{H. van der Holst,} Graphs and obstructions in four dimensions, J. Combin. Theory Ser. B 96:3 (2006), 388--404.}


\newcommand{\hpzn}{\bibitem[HP09]{HP09} \emph{H. van der Holst and R. Pendavingh,} On a graph property generalizing planarity and flatness, Combinatorica, 29 (2009) 337--361.}

\newcommand{\hssf}{\bibitem[HS64]{HS64} \emph{A. Haefliger and B. Steer,} Symmetry of linking coefficients, Comment. Math. Helv. 39 (1964) 259-270.}

\newcommand{\htsf}{\bibitem[HT74]{HT74} \emph{J. Hopcroft and R. E. Tarjan,} Efficient planarity testing, J. of the Association for Computing Machinery, 21:4 (1974) 549--568.}

\newcommand{\hufn}{\bibitem[Hu59]{hu59} * \emph{S. T. Hu,} Homotopy Theory, Academic Press, New York, 1959.}

\newcommand{\huss}{\bibitem[Hu66]{Hu66} * \emph{J.~F.~P.~Hudson,} Extending piecewise linear isotopies, Proc. London Math. Soc. (3) 16 (1966) 651--668.}

\newcommand{\husn}{\bibitem[Hu69]{Hu69} * \emph{J. F. P. Hudson.} Piecewise linear topology, W. A. Benjamin, Inc., New York-Amsterdam, 1969.}


\newcommand{\io}{\bibitem[Io]{Io} * \url{https://en.wikipedia.org/wiki/Category:Impossible_objects}}

\newcommand{\info}{\bibitem[IF]{IF} * \url{http://www.map.mpim-bonn.mpg.de/Intersection_form}}

\newcommand{\irsf}{\bibitem[Ir65]{Ir65} \emph{M.~C.~Irwin,} Embeddings of polyhedral manifolds, Ann. of Math. (2)
82 (1965) 1--14.}

\newcommand{\isot}{\bibitem[Is]{Is} * \url{http://www.map.mpim-bonn.mpg.de/Isotopy}}


\newcommand{\jqnt}{\bibitem[JQ93]{JQ93} * \emph{A. Jaffe, F. Quinn,} ``Theoretical mathematics'': Toward a cultural synthesis of mathematics and theoretical physics. Bull.Am.Math.Soc. 29 (1993) 1-13. arXiv:math/9307227.}

\newcommand{\jozt}{\bibitem[Jo02]{Jo02} \emph{C. M. Johnson.} An obstruction to embedding a simplicial $n$-complex into a $2n$-manifold, Topology Appl. 122:3 (2002) 581--591.}

\newcommand{\jvz}{\bibitem[JVZ]{JVZ} D. Joji\'c, S. T. Vre\'cica, R. T. \v Zivaljevi\' c,
Topology and combinatorics of 'unavoidable complexes', arXiv:1603.08472v1.}


\newcommand{\kalai}{\bibitem[Ka]{Ka} \emph{G. Kalai,} From Oberwolfach: The Topological Tverberg Conjecture is False, `Combinatorics and more' blog post, February 6, 2015, \url{gilkalai.wordpress.com}}

\newcommand{\kano}{\bibitem[Ka91]{Ka91} \emph{G. Kalai,} The diameter of graphs of convex polytopes and $f$-vector theory, Applied geometry and discrete mathematics, DIMACS Ser. Discrete Math. Theoret. Comput. Sci., vol. 4, Amer. Math. Soc., Providence, RI, 1991, pp. 387--411.}

\newcommand{\kefn}{\bibitem[Ke59]{Ke59} \emph{M. Kervaire.} An interpretation of G. Whitehead's generalization of H. Hopf's invariant, Ann. of Math. 62 (1959), 345--362.}

\newcommand{\kh}{\bibitem[Kh]{Kh} \emph{А.И. Храбров.} Руководство по чтению лекций
\url{http://vm.tstu.tver.ru/topics/pdf_tests/lection.pdf}}

\newcommand{\kho}{\bibitem[Kho]{Kho} \emph{N. Khoroshavkina.} A simple characterization of graphs of cutwidth 2, arXiv:1811.06716.}

\newcommand{\kkrot}{\bibitem[KKR]{KKR} \emph{K. Kawarabayashi, Y. Kobayashi and B. Reed.} The disjoint paths problem in quadratic time, J. of Comb. Theory, Ser. B, 102:2 (2012), 424--435.}

\newcommand{\kln}{\bibitem[KLN]{KLN} \emph{ J. Kratochv\'il, A. Lubiw and J. Ne\v set\v ril.} Noncrossing subgraphs in topological layouts, SIAM J. on Discr. Math. 4(2) (1991), 223--244.}

\newcommand{\kmsth}{\bibitem[KM63]{KM63} \emph{M. A. Kervaire and J. W. Milnor,} Groups of homotopy spheres. I,  Ann. of Math. (2) 77 (1963), 504-537.}

\newcommand{\kozeru}{\bibitem[Ko18]{Ko18} * \emph{Е. Колпаков.}
Доказательство теоремы Радона при помощи понижения размерности, Мат. Просвещение, 23 (2018), arXiv:1903.11055.}

\newcommand{\koze}{\bibitem[Ko18]{Ko18} * \emph{E. Kolpakov.}
A proof of Radon Theorem via lowering of dimension, Mat. Prosveschenie, 23 (2018), arXiv:1903.11055.}

\newcommand{\ko}{\bibitem[Ko]{Ko} \emph{E. Kolpakov.} A `converse' to the Constraint Lemma, arXiv:1903.08910.}

\newcommand{\koon}{\bibitem[Ko19]{Ko19} \emph{E. Kogan.} Linking of three triangles in 3-space, arXiv:1908.03865.}

\newcommand{\koto}{\bibitem[Ko21]{Ko21} \emph{E. Kogan.} On the rank of $\Z_2$-matrices with free entries on the diagonal, arXiv:2104.10668.}

\newcommand{\koee}{\bibitem[Ko88]{Ko88} \emph{U. Koschorke.} Link maps and the geometry of their invariants,
Manuscripta Math. 61:4 (1988) 383--415.}

\newcommand{\kono}{\bibitem[Ko91]{Ko91} \emph{U. Koschorke.} Link homotopy with many components,
Topology 30:2 (1991) 267--281.}

\newcommand{\kons}{\bibitem[Ko97]{Ko97} \emph{U. Koschorke.} A generalization of Milnor's $\mu$-invariants to higher-dimensional link maps, Topology 36:2 (1997) 301--324.}

\newcommand{\kps}{\bibitem[KPS]{KPS} * \emph{A. Kaibkhanov, D. Permyakov and A. Skopenkov.}
Realization of graphs with rotation, \url{http://www.turgor.ru/lktg/2005/3/index.htm}.}

\newcommand{\krzz}{\bibitem[Kr00]{Kr00} \emph{V. S. Krushkal.} Embedding obstructions and 4-dimensional thickenings of 2-complexes, Proc. Amer. Math. Soc. 128:12 (2000) 3683--3691. arXiv:math/0004058. }

\newcommand{\ksnn}{\bibitem[KS99]{KS99} * \emph{П. Кожевников и А. Скопенков.} Узкие деревья на плоскости, Мат. Образование. 1999. 2-3. С.~126-131.}

\newcommand{\kstz}{\bibitem[KS20]{KS20} \emph{R. Karasev and A. Skopenkov.}
Some `converses' to intrinsic linking theorems, Discr. Comp. Geom., 70:3 (2023), 921--930, arXiv:2008.02523.}


\newcommand{\ksto}{\bibitem[KS21]{KS21} * \emph{E. Kogan and A. Skopenkov.} A short proof of the Patak-Tancer theorem on non-embeddability of $k$-complexes in $2k$-manifolds,  arXiv:2106.14010.}

\newcommand{\kstoe}{\bibitem[KS21e]{KS21e} \emph{E. Kogan and A. Skopenkov.}
Embeddings of $k$-complexes in $2k$-manifolds and minimum rank of partial symmetric matrices, arXiv:2112.06636v2.}

\newcommand{\kstf}{\bibitem[KS24]{KS24} \emph{R. Karasev and A. Skopenkov.}
Short proofs of Tverberg-type theorems for cell complexes, Discr. Comp. Geom., (2025), arXiv:2405.05629.}


\newcommand{\kutt}{\bibitem[Ku23]{Ku23} \emph{W. K\"uhnel.} Generalized Heawood Numbers, The Electronic Journal of Combinatorics, 30:4 (2023) \#P4.17.}


\newcommand{\kuse}{\bibitem[Ku68]{Ku68} * \emph{К. Куратовский.} Топология. Т.~1,~2. М.: Мир, 1969.}

\newcommand{\kunfo}{\bibitem[Ku94]{Ku94} \emph{W. K\"uhnel.} Manifolds in the skeletons of convex polytopes, tightness, and generalized Heawood inequalities. In Polytopes: abstract, convex and computational (Scarborough, ON, 1993), volume 440 of NATO Adv. Sci. Inst. Ser. C Math. Phys. Sci., pp. 241--247. Kluwer
Acad. Publ., Dordrecht, 1994.}


\newcommand{\kunf}{\bibitem[Ku95]{Ku95} * \emph{W. K\"uhnel}, Tight Polyhedral Submanifolds and Tight Triangulations, Lecture Notes in Math. 1612, Springer, 1995.}

\newcommand{\kytz}{\bibitem[Ky20]{Ky20} \emph{J. Kyn{\v{c}}l,} Simple Realizability of Complete Abstract Topological Graphs Simplified, Discr. Comp. Geom., 64 (2020) 1--27, arXiv:1608.05867.}


\newcommand{\lazz}{\bibitem[La00]{La00} \emph{F. Lasheras.} An obstruction to 3-dimensional thickening,
Proc. Amer. Math. Soc. 2000. 128. P.~893-902.}

\newcommand{\lfma}{\bibitem[LF]{LF} \url{http://www.map.mpim-bonn.mpg.de/Linking_form}}

\newcommand{\lloe}{\bibitem[LL18]{LL18} \emph{A.S. Levine and T. Lidman.} Simply connected, spineless 4-manifolds, Forum of Math., Sigma, 7 (2019) e14, 1--11, arxiv:1803.01765.}

\newcommand{\lo}{\bibitem[Lo]{Lo} M.~de~Longueville. Notes on the topological Tverberg theorem.
Discrete Math.  247 (2002), no.~1--3, 271--297.
(The paper first appeared in
Discrete Math. 241 (2001) 207--233, but the original version suffered from serious publisher's typesetting errors.)}

\newcommand{\loot}{\bibitem[Lo13]{Lo13} \emph{M. de Longueville.} A course in topological combinatorics. Universitext. Springer, New York (2013).}

\newcommand{\lssn}{\bibitem[LS69]{LS69} \emph{W. B. R. Lickorish and L. C. Siebenmann.}
Regular neighborhoods and the stable range,  Trans. Amer. Math. Soc.. 1969. 139. P.~207-230.}

\newcommand{\lsne}{\bibitem[LS98]{LS98} \emph{L. Lovasz and A. Schrijver,}
A Borsuk theorem for antipodal links and a spectral characterization of linklessly embeddable graphs, Proc. Amer. Math. Soc. 126:5 (1998), 1275-1285.}

\newcommand{\ltof}{\bibitem[LT14]{LT14} \emph{E. Lindenstrauss and M. Tsukamoto,} Mean dimension and an embedding problem: an example, Israel J. Math. 199 (2014).}


\newcommand{\lyzf}{\bibitem[LY04]{LY04} * \emph{Y. Lin and A. Yang,} On 3-cutwidth critical graphs, Discrete Mathematics, 275 (2004), 339--346.}

\newcommand{\lz}{\bibitem[LZ]{LZ} * \emph{S. Lando and A. Zvonkin.} Embedded Graphs. Springer.}


\newcommand{\maez}{\bibitem[Ma80]{Ma80} * R. Mandelbaum, {\em Four-Dimensional Topology: An introduction},
Bull. Amer. Math. Soc. (N.S.) 2 (1980) 1-159.}

\newcommand{\mast}{\bibitem[Ma73]{Ma73} \emph{С. В. Матвеев.} Специальные остовы кусочно-линейных многообразий, Мат. Сборник. 1973. 92. С.~282-293.}

\newcommand{\maste}{\bibitem[Ma73]{Ma73} \emph{S. V. Matveev.} Special skeletons of PL manifolds (in Russian), Mat. Sbornik. 1973. 92. P.~282-293.}

\newcommand{\manz}{\bibitem[Ma90]{Ma90} \emph{W. S.  Massey.} Homotopy classification of 3-component links of codimension greater than 2, Topol.  Appl. 34 (1990) 269--300.}

\newcommand{\mans}{\bibitem[Ma97]{Ma97} \emph{Yu. Makarychev.} A short proof of Kuratowski's graph planarity criterion, J. of Graph Theory, 25 (1997), 129--131.}

\newcommand{\matns}{\bibitem[Mat97]{Mat97} \emph{J. Matou\v sek.} A Helly-type theorem for unions of convex sets. Discr. Comp. Geom., 18:1 (1997) 1-12.}

\newcommand{\mazt}{\bibitem[Ma03]{Ma03} * \emph{J.~Matou{\v{s}}ek.} Using the {B}orsuk-{U}lam theorem:
Lectures on topological methods in combinatorics and geometry. Springer Verlag, 2008.}


\newcommand{\mazf}{\bibitem[Ma05]{Ma05} \emph{V. Manturov.} A proof of the Vasiliev conjecture on the planarity of singular links, Izv. RAN 2005.}

\newcommand{\metn}{\bibitem[Me29]{Me29} \emph{K. Menger.} \"Uber pl\"attbare Dreiergraphen und Potenzen nicht pl\"attbarer Graphen, Ergebnisse Math. Kolloq., 2 (1929) 30--31.}

\newcommand{\mezf}{\bibitem[Me04]{Me04} \emph{S. Melikhov.} Sphere eversions and realization of mappings, Trudy MIAN 247 (2004) 159-181 (in Russian) arXiv:math.GT/0305158.}

\newcommand{\mezs}{\bibitem[Me06]{Me06} \emph{S. A. Melikhov}, The van Kampen obstruction and its relatives, 	
Proc. Steklov Inst. Math 266 (2009), 142-176 (= Trudy MIAN 266 (2009), 149-183), arXiv:math/0612082.}

\newcommand{\meoo}{\bibitem[Me11]{Me11} \emph{S. A. Melikhov}, Combinatorics of embeddings, arXiv:1103.5457.}

\newcommand{\meos}{\bibitem[Me17]{Me17} \emph{S. Melikhov,} Gauss type formulas for link map invariants, arXiv:1711.03530.}

\newcommand{\meoe}{\bibitem[Me18]{Me18} \emph{S. A. Melikhov,} A triple-point Whitney trick, J. Topol. Anal., 2018, 1--6. arXiv:2210.04016.}


\newcommand{\metz}{\bibitem[Me20]{Me20} \emph{S. A. Melikhov,} Topological isotopy and Cochran's derived invariants, in `Topology, Geometry, and Dynamics: Rokhlin Memorial', Contemporary Mathematics, 772, AMS, Providence, RI, 2021. arXiv:2011.01409.}

\newcommand{\mett}{\bibitem[Me22]{Me22} \emph{S. A. Melikhov,} Embeddability of joins and products of polyhedra, Topol. Methods in Nonlinear Analysis, 60:1 (2022), 185-201. arXiv:2210.04015.}

\newcommand{\miff}{\bibitem[Mi54]{Mi54} \emph{J. Milnor,} Link groups, Ann. of Math. 59 (1954), 177--195.}

\newcommand{\miso}{\bibitem[Mi61]{Mi61} \emph{J. Milnor,} A procedure for killing homotopy groups of differentiable manifolds, Proc. Sympos. Pure Math, Vol. III (1961), 39--55.}

\newcommand{\mins}{\bibitem[Mi97]{Mi97} \emph{P. Minc.} Embedding simplicial arcs into the plane, Topol. Proc. 1997. 22. 305--340.}


\newcommand{\adnsvr}{\bibitem[MNS]{MNS} * \emph{А. Мирошников, О. Никитенко и А. Скопенков.} Циклы в графах и в гиперграфах: в направлении теории гомологий, Мат. Просвещение, 35 (2025), 137-184, arXiv:2406.16705.}

\newcommand{\dmnse}{\bibitem[MNS]{MNS} * \emph{A. Miroshnikov, O. Nikitenko, A. Skopenkov.}
Cycles in graphs and in hypergraphs: towards homology theory (in Russian), Mat. Prosveschenie, 35 (2025), 137-184, arXiv:2406.16705.}

\newcommand{\moss}{\bibitem[Mo77]{Mo77} * \emph{E. E. Moise.} Geometric Topology in Dimensions 2 and 3 (GTM), Springer-Verlag, 1977.}

\newcommand{\moen}{\bibitem[Mo89]{Mo89} \textit{B. Mohar}. An obstruction to embedding graphs in
surfaces. Discrete Math. 78 (1989) 135--142.}

\newcommand{\moze}{\bibitem[Mo08]{Mo08} \textit{T. Moriyama}. An invariant of embeddings of 3-manifolds in 6-manifolds and Milnor's triple linking number, J. Math. Sci. Univ. Tokyo, 18 (2011), 193--237. arXiv:0806.3733.}


\newcommand{\mrst}{\bibitem[MRS+]{MRS+} \emph{A. de Mesmay, Y. Rieck, E. Sedgwick, M. Tancer,}
Embeddability in $\R^3$ is NP-hard. arXiv:1708.07734.}

\newcommand{\mesczs}{\bibitem[MS06]{MS06} \emph{S.A. Melikhov, E.V. Shchepin,} The telescope approach to embeddability of compacta. arXiv:math.GT/0612085.}

\newcommand{\msos}{\bibitem[MS17]{MS17}  \emph{T. Maciazek, A. Sawicki.} Homology groups for particles on one-connected graphs
J. Math. Phys. 58, 062103 (2017). arXiv:1606.03414.}

\newcommand{\mstwof}{\bibitem[MST+]{MST+} \emph{J. Matou\v sek, E. Sedgwick, M. Tancer, U. Wagner}, Embeddability in the 3-sphere is decidable, Journal of the ACM 65:1 (2018) 1--49, arXiv:1402.0815.}


\newcommand{\mtzo}{\bibitem[MT01]{MT01} * \emph{B. Mohar and C. Thomassen.} Graphs on Surfaces.
The John Hopkins University Press, 2001.}

\newcommand{\mtwoz}{\bibitem[MTW10]{MTW10} \emph{J. Matou\v sek, M. Tancer, U. Wagner.} A geometric proof of
the colored Tverberg theorem, Discr. and Comp. Geometry, 47:2 (2012), 245--265. arXiv:1008.5275.}


\newcommand{\mtwoo}{\bibitem[MTW]{MTW} \emph{J. Matou\v sek, M. Tancer, U. Wagner.}
Hardness of embedding simplicial complexes in $\R^d$, J. Eur. Math. Soc. 13:2 (2011), 259--295. arXiv:0807.0336.}


\newcommand{\musf}{\bibitem[Mu74]{Mu74} \emph{J. Muncres,} Elementary differential topology, Complement to the book by J. W. Milnor and J. D. Stasheff, {\it Characteristic Classes}, Ann. of Math. St. 76 (1974), Princeton Univ. Press, Princeton, NJ.}

\newcommand{\mwoe}{\bibitem[MW18]{MW18} * \emph{F. Manin, S. Weinberger.} Algorithmic aspects of immersibility and embeddability, Intern. Math. Res. Notices, rnae170. arXiv:1812.09413.}


\newcommand{\mwsn}{\bibitem[MW69]{MW69} * \emph{J. MacWilliams}. Orthogonal matrices over finite fields. Amer. Math. Monthly, 76 (1969) 152--164.}

\newcommand{\mwofo}{\bibitem[MW14]{MW14} \emph{I. Mabillard and U. Wagner.} Eliminating Tverberg Points, I. An Analogue of the Whitney Trick, Proc. of the 30th Annual Symp. on Comp. Geom. (SoCG'14), ACM, New York, 2014, pp. 171--180.}

\newcommand{\mwof}{\bibitem[MW15]{MW15} \emph{I. Mabillard and U. Wagner.}
Eliminating Higher-Multiplicity Intersections, I. A Whitney Trick for Tverberg-Type Problems. arXiv:1508.02349.}


\newcommand{\mwos}{\bibitem[MW16]{MW16} \emph{I. Mabillard and U. Wagner.} Eliminating Higher-Multiplicity Intersections, II. The Deleted Product Criterion in the $r$-Metastable Range. arXiv:1601.00876v2.}

\newcommand{\mwosd}{\bibitem[MW16']{MW16'} \emph{I. Mabillard and U. Wagner.} Eliminating Higher-Multiplicity Intersections, II. The Deleted Product Criterion in the r-Metastable Range,
Proceedings of the 32nd Annual Symposium on Computational Geometry (SoCG'16).}


\newcommand{\neno}{\bibitem[Ne91]{Ne91} \emph{S. Negami.} Ramsey theorems for knots, links and spatial graphs,
Trans. Amer. Math. Soc., 324 (1991), 527--541.}



\newcommand{\nizz}{\bibitem[Ni00]{Ni00} \emph{R. Nikkuni.} The second skew-symmetric cohomology group and spatial embeddings of graphs, J. Knot Theory Ram. 9 (2000), 387--411.}

\newcommand{\nkon}{\bibitem[NKS]{NKS} * \emph{L. T. Nguyen, J. Kim, B. Shim.}
Low-Rank Matrix Completion: A Contemporary Survey. arXiv:1907.11705.}

\newcommand{\noss}{\bibitem[No76]{No76} * \emph{С. П. Новиков.} Топология-1. М.: Наука, 1976. (Итоги науки и техники. ВИНИТИ. Современные проблемы математики. Основные направления, 12).}

\newcommand{\nszn}{\bibitem[NS09]{NS09} \emph{I. Novik and E. Swartz,} Socles of Buchsbaum modules, complexes and posets, Adv. Math. 222 (2009), 2059-2084. arXiv:0711.0783.}

\newcommand{\nwns}{\bibitem[NW97]{NW97} \emph{A. Nabutovsky, S. Weinberger}. Algorithmic aspects of homeomorphism problems. arXiv:math/9707232.}


\newcommand{\omoe}{\bibitem[Om18]{Om18} * \emph{А. Омельченко,} Теория графов. М.: МЦНМО, 2018.}

\newcommand{\orszo}{\bibitem[ORS]{ORS} \emph{A. Onischenko, D. Repov\v s and A. Skopenkov.}
Resolutions of 2-polyhedra by fake surfaces and embeddings into $\R^4$, Contemp. Math.  288 (2001) 396--400.}

\newcommand{\ossf}{\bibitem[OS74]{OS74} \emph{R. P. Osborne and R. S. Stevens.} Group presentations
corresponding to spines of 3-manifolds, I, Amer. J.~Math. 1974. 96. P.~454-471; II, Amer. J.~Math. 1977. 234.
P.~213-243; III, Amer. J.~Math. 1977. 234 P.~245-251.}


\newcommand{\oz}{\bibitem[Oz]{Oz} \emph{M. \"Ozaydin,} Equivariant maps for the symmetric group, unpublished,
\url{http://minds.wisconsin.edu/handle/1793/63829}.}

\newcommand{\panof}{\bibitem[Pan15]{Pan15} \emph{K. Panagiotis.} A note on the topology of irreducible $SO(3)$-manifolds, 	arXiv:1508.06150.}

\newcommand{\paof}{\bibitem[Pa15]{Pa15} \emph{S. Parsa,} On links of vertices in simplicial $d$-complexes embeddable in the Euclidean $2d$-space, Discrete Comput. Geom. 59:3 (2018), 663--679.
This is arXiv:1512.05164v4 up to numbering of sections, theorems etc.; we refer to numbering in arxiv version.
Correction: Discrete Comput. Geom. 64:3 (2020) 227--228.}

\newcommand{\paoe}{\bibitem[Pa18]{Pa18} \emph{S. Parsa,} On links of vertices in simplicial $d$-complexes
embeddable in the euclidean $2d$-space, arXiv:1512.05164v6.}

\newcommand{\patz}{\bibitem[Pa20]{Pa20} \emph{S. Parsa,} On links of vertices in simplicial $d$-complexes
embeddable in the euclidean $2d$-space, arXiv:1512.05164v8.}


\newcommand{\patzl}{\bibitem[Pa20]{Pa20} \emph{S. Parsa,}
Correction to: On the Links of Vertices in Simplicial $d$-Complexes Embeddable in the Euclidean $2d$-Space,
Discrete Comput. Geom. 64:3 (2020) 227--228.}

\newcommand{\patza}{\bibitem[Pa20]{Pa20} \emph{S. Parsa,} On the Smith classes, the van Kampen obstruction and embeddability of $[3]*K$, arXiv:2001.06478.}

\newcommand{\patzb}{\bibitem[Pa20b]{Pa20b} \emph{S. Parsa,} On the embeddability of $[3]*K$, arXiv:2001.06506.}

\newcommand{\pato}{\bibitem[Pa21]{Pa21} \emph{S. Parsa,} Instability of the Smith index under joins and applications to embeddability, Trans. Amer. Math. Soc. 375 (2022), 7149--7185, arXiv:2103.02563.}

\newcommand{\pak}{\bibitem[Pa]{Pa} * \emph{I. Pak}, Lectures on Discrete and Polyhedral Geometry, \url{http://www.math.ucla.edu/~pak/geompol8.pdf}.}

\newcommand{\peze}{\bibitem[Pe08]{Pe08} \emph{Д. Пермяков.} Классификация погружений графов в плоскость,
Вестник МГУ, сер.1, 2008, N5, 55-56.}

\newcommand{\peos}{\bibitem[Pe16]{Pe16} \emph{Д. Пермяков.} Матем. сб., 207:6 (2016),  93--112.}

\newcommand{\pest}{\bibitem[Pe72]{Pe72} * \emph{B. B. Peterson.} The Geometry of Radon's Theorem, Amer. Math. Monthly 79 (1972), 949-963.}


\newcommand{\prnf}{\bibitem[Pr95]{Pr95} * \emph{V. V. Prasolov.} Intuitive topology. Amer. Math. Soc., Providence, R.I., 1995.}

\newcommand{\prnfr}{\bibitem[Pr95]{Pr95} * \emph{В. В. Прасолов.} Наглядная топология. М.: МЦНМО, 1995.}


\newcommand{\przs}{\bibitem[Pr06]{Pr06} * \emph{V. V. Prasolov.}
Elements of Combinatorial and Differential Topology, 2006, GSM 74, Amer. Math. Soc., Providence, RI.}

\newcommand{\przsru}{\bibitem[Pr04]{Pr04} * \emph{В. В. Прасолов.}
Элементы комбинаторной и дифференциальной топологии. М.: МЦНМО, 2004. \url{http://www.mccme.ru/prasolov}.}

\newcommand{\przse}{\bibitem[Pr07]{Pr07} * \emph{V. V. Prasolov.} Elements of homology theory. 2007, GSM 74, Amer. Math. Soc., Providence, RI.}


\newcommand{\przseru}{\bibitem[Pr06]{Pr06} * \emph{В. В. Прасолов.} Элементы теории гомологий. М.: МЦНМО, 2006.}


\newcommand{\psns}{\bibitem[PS96]{PS96} * \emph{V. V. Prasolov, A. B. Sossinsky } Knots, Links, Braids, and 3-manifolds. Amer. Math. Soc. Publ., Providence, R.I., 1996.}


\newcommand{\pszf}{\bibitem[PS05]{PS05} * \emph{В. В. Прасолов и М. Б. Скопенков.}
Рамсеевская теория зацеплений, Мат. Просвещение. 2005. 9. С.~108--115.}

\newcommand{\pszfen}{\bibitem[PS05]{PS05} * \emph{V. V. Prasolov and M.B. Skopenkov.}
Ramsey link theory, Mat, Prosvescheniye, 9 (2005), 108--115.}

\newcommand{\psoo}{\bibitem[PS11]{PS11} \emph{Y. Ponty and C. Saule.} A combinatorial framework for designing (pseudoknotted) RNA algorithms, Proc. of the 11th Intern. Workshop on Algorithms in Bioinformatics, WABI'11, 250--269.}


\newcommand{\pstz}{\bibitem[PS20]{PS20} \emph{S. Parsa and A. Skopenkov.} On embeddability of joins and their `factors', Topol. Appl., 326 (2023) 108409, arXiv:2003.12285.}



\newcommand{\psszn}{\bibitem[PSS]{PSS} \emph{M. J. Pelsmajer, M. Schaefer and D. Stasi.} Strong Hanani-Tutte on the projective plane. SIAM J. Discrete Math., 23:3 (2009) 1317--1323.}

\newcommand{\psszs}{\bibitem[PSS]{PSS} \emph{M. J. Pelsmajer, M. Schaefer, and D. \v Stefankovi\v c.}
Removing even crossings. J. Combin. Theory Ser. B, 97(4):489–500, 2007.}

\newcommand{\pton}{\bibitem[PT19]{PT19} \emph{P. Pat\'ak and M. Tancer.} Embeddings of $k$-complexes into $2k$-manifolds. Discrete Comput. Geom. 71 (2024), 960--991. arXiv:1904.02404v4.}

\newcommand{\pw}{\bibitem[PW]{PW} \emph{I. Pak, S. Wilson}, G\lowercase{EOMETRIC REALIZATIONS OF POLYHEDRAL COMPLEXES}, \linebreak \url{http://www.math.ucla.edu/~pak/papers/Fary-full31.pdf}.}


\newcommand{\razf}{\bibitem[RA05]{RA05} * \emph{J. L. Ram\'irez Alfons\'in.} Knots and links in spatial graphs: a survey. Discrete Math., 302 (2005), 225--242.}

\newcommand{\rep}{\bibitem[Rep]{Rep} Referee's report on the paper ``Some `converses' to intrinsic linking theorems', \url{https://www.mccme.ru/circles/oim/materials/ksreport.pdf}}

\newcommand{\rnoo}{\bibitem[RN11]{RN11} * \emph{R. L. Ricca, B. Nipoti.} Gauss' linking number revisited.
J. of Knot Theory and Its Ramif. 20:10 (2011) 1325--1343. \url{https://www.maths.ed.ac.uk/~v1ranick/papers/ricca.pdf} .}

\newcommand{\rrstz}{\bibitem[RRS]{RRS} * \emph{V. Retinskiy, A. Ryabichev and A. Skopenkov.}
Motivated exposition of the proof of the Tverberg Theorem (in Russian).
Mat. Prosveschenie, 27 (2021), 166--169. arXiv:2008.08361.}


\newcommand{\rssec}{\bibitem[RS68]{RS68} \emph{C. P. Rourke and B. J. Sanderson,} Block bundles II, Ann. of Math. (2), 87 (1968) 431--483.}

\newcommand{\rsst}{\bibitem[RS72]{RS72} * \emph{C. P. Rourke and B. J. Sanderson,}
\newblock Introduction to Piecewise-Linear Topology,
\newblock \emph{Ergebn.\ der Math.} 69, Springer-Verlag, Berlin, 1972.}

\newcommand{\rsstr}{\bibitem[RS72]{RS72} * \emph{К. П. Рурк и Б. Дж. Сандерсон.} Введение в кусочно-линейную топологию, Москва. Мир. 1974.}

\newcommand{\rsns}{\bibitem[RS96]{RS96} * \emph{D. Repov\v s and A. B. Skopenkov.}
Embeddability and isotopy of polyhedra in Euclidean spaces,
Proc. of the Steklov Inst. Math. 1996. 212. P.~173-188.}

\newcommand{\rsne}{\bibitem[RS98]{RS98} \emph{D. Repov\v s and A. B. Skopenkov.}
A deleted product criterion for approximability of a map by embeddings, Topol. Appl. 1998. 87 P.~1-19.}

\newcommand{\rsnn}{\bibitem[RS99]{RS99} * \emph{D. Repov\v s and A. B. Skopenkov.} New results on embeddings of polyhedra and manifolds into Euclidean spaces,
Russ. Math. Surv. 54:6 (1999), 1149--1196.}


\newcommand{\rsnnd}{\bibitem[RS99']{RS99'} * \emph{Д. Реповш и А. Скопенков.}
Кольца Борромео и препятствия к вложимости, Труды МИРАН. 1999. 225. С.~331-338.}

\newcommand{\rszz}{\bibitem[RS00]{RS00} \emph{D. Repov\v s and A. Skopenkov.} Cell-like resolutions of polyhedra by special ones,  Colloq. Math. 2000. 86:2. P. 231--237.}

\newcommand{\rszzd}{\bibitem[RS00']{RS00'} * \emph{Д. Реповш и А. Скопенков.} Характеристические классы для начинающих, Мат. Просвещение. 2000. 4. С.~151-176.}

\newcommand{\rszo}{\bibitem[RS01]{RS01} \emph{D. Repovs and A. Skopenkov.} On contractible $n$-dimensional compacta, non-embeddable into $\R^{2n}$, Proc. Amer. Math. Soc. 129 (2001) 627--628.}

\newcommand{\rszt}{\bibitem[RS02]{RS02} * \emph{Д. Реповш и А. Скопенков.} Теория препятствий для начинающих,
Мат. Просвещение. 2002. 6. C.~60-77.}

\newcommand{\rszf}{\bibitem[RS04]{RS04} \emph{N. Robertson and P. Seymour.} Graph Minors. XX. Wagner's conjecture, J. of Comb. Theory, B, 92:2 (2004) 325--357.}

\newcommand{\rssnf}{\bibitem[RSS]{RSS95} \emph{D. Repov\v s, A. B. Skopenkov  and E. V. \v S\v cepin.}
On uncountable collections of continua and their span, Colloq. Math. 1995. 69:2. P.~289-296.}

\newcommand{\rssnfd}{\bibitem[RSS']{RSS95'} \emph{D. Repov\v s, A. B. Skopenkov and E. V \v S\v cepin.}
On embeddability of $X\times I$ into Euclidean space, Houston J.~Math. 1995. 21. P.~199-204.}

\newcommand{\rssz}{\bibitem[RSS+]{RSSZ} * \emph{A. Rukhovich, A. Skopenkov, M. Skopenkov, A. Zimin},
Realizability of hypergraphs, \url{https://www.turgor.ru/lktg/2013/1/1-1en.pdf} .}


\newcommand{\rstnt}{\bibitem[RST']{RST93} \emph{N. Robertson, P. Seymour and R. Thomas}, Linkless embeddings of graphs in 3-space, Bull. of the Amer. Math. Soc., 21 (1993) 84--89.}

\newcommand{\rstno}{\bibitem[RST]{RST91} * \emph{N. Robertson, P. Seymour and R. Thomas}, A survey of
linkless embeddings, Graph Structure Theory (Seattle, WA, 1991), Contemp. Math. 147, (1993) 125--136.}


\newcommand{\rwzl}{\bibitem[RWZ+]{RWZ+} \emph{Y. Ren, C. Wen, S. Zhen, N. Lei, F. Luo, D.X. Gu},
Characteristic class of isotopy for surfaces, J. Syst. Sci. Complex. 33 (2020) 2139--2156.}


\newcommand{\saeo}{\bibitem[Sa81]{Sa81} \emph{H. Sachs.} On spatial representation of finite graphs,
in: Finite and infinite sets (Eger, 1981), 649--662, Colloq. Math. Soc. Janos Bolyai, 37, North-Holland, Amsterdam, 1984.}

\newcommand{\sano}{\bibitem[Sa91]{Sa91} \emph{K. S. Sarkaria.}
A one-dimensional Whitney trick and Kuratowski's graph planarity criterion, Israel J.~Math. 73 (1991), 79--89.}


\newcommand{\sanov}{\bibitem[Sa91g]{Sa91g} \emph{K. S. Sarkaria.} A generalized Van Kampen-Flores theorem, Proc. Amer. Math. Soc. 111 (1991), 559--565.}

\newcommand{\sant}{\bibitem[Sa92]{Sa92} \emph{K. S. Sarkaria.} Tverberg’s theorem via number fields. Israel J. Math., 79:317–320, 1992.}

\newcommand{\sann}{\bibitem[Sa99]{Sa99} O. Saeki {\em On punctured 3-manifolds in 5-sphere}, Hiroshima Math. J. 29 (1999) 255--272.}


\newcommand{\sazz}{\bibitem[Sa00]{Sa00} \emph{K. S. Sarkaria.} Tverberg partitions and Borsuk-Ulam theorems. Pacific J. Math., 196:1 (2000) 231--241.}

\newcommand{\sczf}{\bibitem[Sc04]{Sc04} \emph{T. Sch\"oneborn.} On the Topological Tverberg Theorem, arXiv:math/0405393.}


\newcommand{\scot}{\bibitem[Sc13]{Sc13} * \emph{M. Schaefer.} Hanani-Tutte and related results. In Geometry --- intuitive, discrete, and convex, Bolyai Soc. Math. Stud., 24 (2013), 259--299.
\url{http://ovid.cs.depaul.edu/documents/htsurvey.pdf} }


\newcommand{\sctz}{\bibitem[Sc20]{Sc20} \emph{M. Schaefer.} The Graph Crossing Number and
its Variants: A Survey. The Electr. J. of Comb. (2020), DS21, \url{https://www.combinatorics.org/files/Surveys/ds21/ds21v5-2020.pdf}}


\newcommand{\scef}{\bibitem[Sc84]{Sc84} \emph{E.~V.~\v S\v cepin.} Soft mappings of manifolds, Russian Math. Surveys, 39:5 (1984).}

\newcommand{\shfs}{\bibitem[Sh57]{Sh57} \emph{A. Shapiro,} Obstructions to the embedding of a complex in a Euclidean space, I, The first obstruction, Ann. Math. 66 (1957), 256--269.}


\newcommand{\shen}{\bibitem[Sh89]{Sh89} * \emph{Ю. А. Шашкин,} Неподвижные точки, М., Наука, 1989.}

\newcommand{\shoe}{\bibitem[Sh18]{Sh18} * \emph{S. Shlosman},  Topological Tverberg Theorem: the proofs and the counterexamples, Russian Math. Surveys, 73:2 (2018), 175–182. arXiv:1804.03120.}

\newcommand{\sisn}{\bibitem[Si69]{Si69} \emph{K. Sieklucki.} Realization of mappings, Fund. Math. 1969. 65. P.~325-343.}

\newcommand{\sios}{\bibitem[Si16]{Si16} \emph{S. Simon,} Average-Value Tverberg Partitions via Finite Fourier Analysis, Israel J. Math., 216 (2016), 891-904, arXiv:1501.04612.}
 

\newcommand{\sknf}{\bibitem[Sk94]{Sk94} \emph{А. Скопенков.} Геометрическое доказательство теоремы
Нойвирта об утолщаемости 2-мерных полиэдров, Math. Notes. 1995. 58:5. P.~1244-1247.}


\newcommand{\skns}{\bibitem[Sk97]{Sk97} \emph{A. Skopenkov,} On the deleted product criterion for embeddability of manifolds in $\R^m$, Comment. Math. Helv. 72 (1997), 543--555.}

\newcommand{\skne}{\bibitem[Sk98]{Sk98} \emph{A. B. Skopenkov.} On the deleted product criterion for embeddability in $\R^m$, Proc. Amer. Math. Soc., 126:8 (1998), 2467-2476.}

\newcommand{\skzz}{\bibitem[Sk00]{Sk00} \emph{A. Skopenkov,} On the generalized Massey--Rolfsen invariant for link maps, Fund. Math. 165 (2000), 1--15.}

\newcommand{\skzt}{\bibitem[Sk02]{Sk02} \emph{A. Skopenkov,} On the Haefliger-Hirsch-Wu invariants for embeddings and immersions, Comment. Math. Helv. 77 (2002), 78--124.}

\newcommand{\skzth}{\bibitem[Sk03]{Sk03} \emph{M. Skopenkov,} Embedding products of graphs into Euclidean spaces,
Fund. Math. 179 (2003),~191--198, arXiv:0808.1199.}

\newcommand{\skzthd}{\bibitem[Sk03']{Sk03'} \emph{M. Skopenkov,} On approximability by embeddings of cycles in the plane, Topol. Appl. 134 (2003),~1--22, arXiv:0808.1187.}

\newcommand{\skzf}{\bibitem[Sk05]{Sk05} * \emph{A. Skopenkov,}
On the Kuratowski graph planarity criterion, Mat. Prosveschenie, 9 (2005), 116-128. arXiv:0802.3820.}


\newcommand{\skzs}{\bibitem[Sk05i]{Sk05i} \emph{A. Skopenkov,} A new invariant and parametric connected sum of embeddings, Fund. Math. 197 (2007) 253--269. arxiv:math/0509621.}

\newcommand{\skzei}{\bibitem[Sk05]{Sk05} \emph{A.  Skopenkov,} A classification of smooth embeddings of
4-manifolds in 7-space, I, Topol. Appl., 157 (2010) 2094--2110. arXiv:math/0512594.}

\newcommand{\skze}{\bibitem[Sk06]{Sk06} * \emph{A. Skopenkov,} Embedding and knotting of manifolds in Euclidean spaces, London Math. Soc. Lect. Notes, 347 (2008) 248--342. arXiv:math/0604045.}

\newcommand{\skzsi}{\bibitem[Sk06']{Sk06'} \emph{A. Skopenkov,} A classification of smooth embeddings of 3-manifolds in 6-space, Math. Zeitschrift, 260:3 (2008) 647--672. arxiv:math/0603429.}

\newcommand{\skzsc}{\bibitem[Sk06c]{Sk06c} \emph{A. Skopenkov,} Classification of embeddings below the metastable dimension, arXiv:math/0607422.}

\newcommand{\skozp}{\bibitem[Sk08]{Sk08} \emph{A.  Skopenkov,} Embeddings of $k$-connected $n$-manifolds into
$\R^{2n-k-1}$. arxiv:math/0812.0263; earlier version published in Proc. Amer. Math. Soc., 138 (2010) 3377--3389.}

\newcommand{\skoz}{\bibitem[Sk10]{Sk10} * \emph{А. Скопенков,} Вложения в плоскость графов с вершинами степени 4,
Мат. просвещение, 21 (2017), arXiv:1008.4940.}

\newcommand{\skoo}{\bibitem[Sk11]{Sk11} \emph{M. Skopenkov,} When is the set of embeddings finite up to isotopy? Intern. J. Math. 26:7 (2015), 28 pp. arXiv:1106.1878.}

\newcommand{\skofo}{\bibitem[Sk14]{Sk14} \emph{A. Skopenkov,} How do autodiffeomorphisms act on embeddings, Proc. A of the Royal Society of Edinburgh, 148:4 (2018), 835--848. arXiv:1402.1853.}

\newcommand{\sks}{\bibitem[Sk14]{Sk14} * \emph{A. Skopenkov,} Realizability of hypergraphs and intrinsic linking  theory, Mat. Prosveschenie, 32 (2024), 125--159, arXiv:1402.0658.}

\newcommand{\sksr}{\bibitem[Sk14]{Sk14} * \emph{А. Скопенков,} Реализуемость гиперграфов и неотъемлемая зацепленность, Мат. просвещение, 32 (2024), 125--159. arXiv:1402.0658.}


\newcommand{\skof}{\bibitem[Sk15]{Sk15} * \emph{А. Скопенков,} Алгебраическая топология с геометрической точки зрения, Москва, МЦНМО, 2015 (1е издание).}

\newcommand{\skofe}{\bibitem[Sk15]{Sk15} * \emph{A. Skopenkov,} Algebraic Topology From Geometric Viewpoint (in Russian), MCCME, Moscow, 2015 (1st edition). }

\newcommand{\skofel}{\bibitem[Sk15e]{Sk15e} * \emph{А. Скопенков,} Алгебраическая топология
с геометрической точки зрения, эл. версия, \url{http://www.mccme.ru/circles/oim/home/combtop13.htm\#photo}}


\newcommand{\skoomp}{\bibitem[Sk11]{Sk11} A. Skopenkov, A simple proof of the Abel-Ruffini theorem (in Russian),
Mat. Prosveschenie, 15 (2011) 113-126, arXiv:1102.2100.}

\newcommand{\skofmp}{\bibitem[Sk15]{Sk15} A. Skopenkov, A short elementary proof of the Ruffini-Abel Theorem (in Russian),
Mat. Prosveschenie, 36 (2026) 95--113. Abridged English version is published in [Sk21m, \S8]; full English version: arXiv:1508.03317.}

\newcommand{\skotzr}{\bibitem[Sk20]{Sk20} * \emph{А. Скопенков,} Алгебраическая топология с геометрической точки зрения, Москва, МЦНМО, 2020 (2е издание).
Обновляемая версия части книги: \url{http://www.mccme.ru/circles/oim/obstruct.pdf}}

\newcommand{\skotz}{\bibitem[Sk20]{Sk20} * \emph{A. Skopenkov,} Algebraic Topology From Geometric Standpoint (in Russian), MCCME, Moscow, 2020 (2nd edition).
Update of a part: \url{http://www.mccme.ru/circles/oim/obstruct.pdf} .
Part of the English translation: \url{https://www.mccme.ru/circles/oim/obstructeng.pdf}.}


\newcommand{\skofp}{\bibitem[Sk15]{Sk15} \emph{A. Skopenkov,} Classification of knotted tori,
Proc. A of the Royal Soc. of Edinburgh, 150:2 (2020), 549-567. Full version: arXiv:1502.04470.}


\newcommand{\skos}{\bibitem[Sk16]{Sk16} * \emph{A. Skopenkov,} A user's guide to the topological Tverberg Conjecture, arXiv:1605.05141v5. Abridged earlier published version: Russian Math. Surveys, 73:2 (2018), 323--353.}


\newcommand{\skosd}{\bibitem[Sk16']{Sk16'} * \emph{A. Skopenkov,} Stability of intersections of graphs in the plane and the van Kampen obstruction, Topol. Appl. 240(2018) 259--269, arXiv:1609.03727.}


\newcommand{\skosc}{\bibitem[Sk16c]{Sk16c} * \emph{A. Skopenkov,}  Embeddings in Euclidean space: an introduction to their classification, to appear in Boll. Man. Atl. 
\url{http://www.map.mpim-bonn.mpg.de/}
(this site is under long maintenance),
\url{https://old.mccme.ru/circles/oim/matlas-emb/Embeddings_in_Euclidean_space:_an_introduction_to_their_classification.html}
 (some Manifold Atlas features are not available on this page).}

\newcommand{\skosie}{\bibitem[Sk16e]{Sk16e} * \emph{A. Skopenkov,} Embeddings just below the stable range: classification, to appear in Boll. Man. Atl.
\url{http://www.map.mpim-bonn.mpg.de/Embeddings_just_below_the_stable_range:_classification}}

\newcommand{\skost}{\bibitem[Sk16t]{Sk16t} * \emph{A. Skopenkov,} 3-manifolds in 6-space, to appear in Boll. Man. Atl. \url{http://www.map.mpim-bonn.mpg.de/3-manifolds_in_6-space}.}

\newcommand{\skosf}{\bibitem[Sk16f]{Sk16f} * \emph{A. Skopenkov,} 4-manifolds in 7-space, to appear in Boll. Man. Atl. \url{http://www.map.mpim-bonn.mpg.de/4-manifolds_in_7-space}.}

\newcommand{\skosh}{\bibitem[Sk16h]{Sk16h} * \emph{A. Skopenkov,} High codimension links, to appear in Boll. Man. Atl.
\linebreak
\url{http://www.map.mpim-bonn.mpg.de/High_codimension_links}.}

\newcommand{\skosi}{\bibitem[Sk16i]{Sk16i} * \emph{A. Skopenkov,} Isotopy, submitted to Boll. Man. Atl.
\url{http://www.map.mpim-bonn.mpg.de/Isotopy}.}

\newcommand{\skosk}{\bibitem[Sk16k]{Sk16k} * \emph{A. Skopenkov,} Knotted tori,
\url{http://www.map.mpim-bonn.mpg.de/Knotted_tori}.}

\newcommand{\skoss}{\bibitem[Sk16s]{Sk16s} * \emph{A. Skopenkov,} Knots, i.e. embeddings of spheres,
\linebreak
\url{http://www.map.mpim-bonn.mpg.de/Knots,_i.e._embeddings_of_spheres}.}

\newcommand{\skose}{\bibitem[Sk17]{Sk17} \emph{A. Skopenkov,}
Eliminating higher-multiplicity intersections in the metastable dimension range. arXiv:1704.00143.}

\newcommand{\skosed}{\bibitem[Sk17v]{Sk17v} * \emph{A. Skopenkov,}
On van Kampen-Flores, Conway-Gordon-Sachs and Radon theorems,  arXiv:1704.00300.}

\newcommand{\sk}{\bibitem[Sk17o]{Sk17o} \emph{A. Skopenkov,} On the metastable Mabillard-Wagner conjecture.  arXiv:1702.04259.}

\newcommand{\skmos}{\bibitem[Sk17d]{Sk17d} \emph{M. Skopenkov}. Discrete field theory: symmetries and conservation laws, arXiv:1709.04788.}

\newcommand{\skoe}{\bibitem[Sk18]{Sk18} * \emph{A. Skopenkov.} Invariants of graph drawings in the plane.
Arnold Math. J., 6 (2020) 21--55; full version: arXiv:1805.10237.}


\newcommand{\skoer}{\bibitem[Sk18]{Sk18} * \emph{А. Скопенков,} Инварианты изображений графов на плоскости,
Мат. просвещение, 31 (2023), 74-127. arXiv:1805.10237.}


\newcommand{\sktthd}{\bibitem[Sk23']{Sk23'} * \emph{A. Skopenkov.} Invariants of graph drawings in the plane (in Russian). Mat. Prosveschenie, 31 (2023), 74-127. arXiv:1805.10237.}

\newcommand{\skoeo}{\bibitem[Sk18o]{Sk18o} * \emph{A. Skopenkov.} A short exposition of S. Parsa's theorems on intrinsic linking and non-realizability. Discr. Comp. Geom. 65:2 (2021), 584--585; full version:  arXiv:1808.08363.}


\newcommand{\skona}{\bibitem[Sk19]{Sk19} * \emph{A. Skopenkov,} A short exposition of the Levine-Lidman example of spineless 4-manifolds, arXiv:1911.07330.}

\newcommand{\sktze}{\bibitem[Sk21m]{Sk21m} * \emph{A. Skopenkov.} Mathematics via Problems. Part 1: Algebra. Amer. Math. Soc., Providence, 2021. Preliminary version: \url{https://www.mccme.ru/circles/oim/algebra_eng.pdf}}

\newcommand{\sktz}{\bibitem[Sk20u]{Sk20u} * \emph{A. Skopenkov.} A user's guide to basic knot and link theory,
in: Topology, Geometry, and Dynamics, Contemporary Mathematics, vol. 772, Amer. Math. Soc., Providence, RI, 2021, pp. 281--309.
Russian version: Mat. Prosveschenie 27 (2021), 128--165. arXiv:2001.01472.}

\newcommand{\sktzru}{\bibitem[Sk20u]{Sk20u} * \emph{А. Скопенков.} Основы теории узлов и зацеплений для пользователя, Мат. просвещение, 27 (2021), 128--165. arXiv:2001.01472.}

\newcommand{\sktzo}{\bibitem[Sk20o]{Sk20o} \emph{A. Skopenkov.} On some results of S. Abramyan and T. Panov, arXiv:2005.11152.}

\newcommand{\sktzr}{\bibitem[Sk20e]{Sk20e} * \emph{A. Skopenkov.}
Extendability of simplicial maps is undecidable, Discr. Comp. Geom., 69:1 (2023), 250--259, arXiv:2008.00492.}


\newcommand{\sktzd}{\bibitem[Sk21d]{Sk21d} * \emph{A. Skopenkov.}
On different reliability standards in current mathematical research, arXiv:2101.03745.
More often updated version: \url{https://www.mccme.ru/circles/oim/rese_inte.pdf}.}

\newcommand{\sktt}{\bibitem[Sk22]{Sk22} * \emph{A. Skopenkov.} Invariants of embeddings of 2-surfaces in 3-space,
arXiv:2201.10944.}

\newcommand{\skttn}{\bibitem[Sk22n]{Sk22n} * \emph{A. Skopenkov}, Netflix problem and realization of (hyper)graphs, \url{https://www.mccme.ru/circles/oim/home/netflix20sep.pdf}}

\newcommand{\sktth}{\bibitem[Sk23]{Sk23} \emph{A. Skopenkov.}  To S. Parsa's theorem on embeddability of joins, arXiv:2302.11537.}

\newcommand{\sktf}{\bibitem[Sk24]{Sk24} * \emph{A. Skopenkov.} Double and triple linking numbers in space (in Russian). Mat. Prosveschenie, 33 (2024), 87--132.}

\newcommand{\sktfr}{\bibitem[Sk24]{Sk24} * \emph{А. Скопенков.} Двойные и тройные коэффициенты зацепления в пространстве. Мат. просвещение, 33 (2024), 87--132.}

\newcommand{\sktfb}{\bibitem[Sk24]{Sk24} \emph{A. Skopenkov.} The band connected sum and the second Kirby move for higher-dimensional links, Stud. Sci. Math. Hung., 62:4 (2025) 320--335, arXiv:2406.15367.}


\newcommand{\sktfe}{\bibitem[Sk24]{Sk24} \emph{A. Skopenkov.}
Embeddings of $k$-complexes in $2k$-manifolds and minimum rank of partial symmetric matrices, arXiv:2112.06636v4.}

\newcommand{\skd}{\bibitem[Sk]{Sk} * \emph{А. Скопенков.} Алгебраическая топология с алгоритмической точки зрения, 
\url{http://www.mccme.ru/circles/oim/algor.pdf}.}

\newcommand{\skde}{\bibitem[Sk]{Sk} * \emph{A. Skopenkov.} Algebraic Topology From Algorithmic Standpoint, draft of a book, mostly in Russian,
\url{http://www.mccme.ru/circles/oim/algor.pdf}.}


\newcommand{\skon}{\bibitem[Skw]{Skw} * \emph{A. Skopenkov.} Whitney trick for eliminating multiple intersections, slides for talks at St Petersburg, Brno, Kiev, Moscow,  \url{https://www.mccme.ru/circles/oim/eliminat_talk.pdf}.}

\newcommand{\skl}{\bibitem[EEF]{EEF} * {\it Proposed by D. Eliseev, A. Enne, M. Fedorov, A. Glebov, N. Khoroshavkina, E. Morozov, A. Skopenkov, R. \v Zivaljevi\'c.}
A user's guide to knot and link theory, \url{https://www.turgor.ru/lktg/2019/3} .}

\newcommand{\skr}{\bibitem[Skr]{Skr} * \emph{A. Skopenkov.} Realizability of hypergraphs, slides for talks,  \url{https://www.mccme.ru/circles/oim/algor1_beamer.pdf}.}


\newcommand{\skt}{\bibitem[Skt]{Skt} * \emph{A. Skopenkov.} Transparent anonymous peer review,
\linebreak
\url{https://www.mccme.ru/circles/oim/home/transp_peer_review.htm} .}

\newcommand{\rslktg}{\bibitem[KRR+]{RRSl} * Towards higher-dimensional combinatorial geometry, presented by
E. Kogan, V. Retinskiy, E. Riabov and A. Skopenkov, \url{https://www.mccme.ru/circles/oim/multicomb.pdf} .}



\newcommand{\sm}{\bibitem[Sm]{Sm} S. Smirnov.}

\newcommand{\sper}{\bibitem[Sp]{Sp} * Sperner's lemma defeats the rental harmony problem, \url{https://www.youtube.com/watch?v=7s-YM-kcKME}.}

\newcommand{\sset}{\bibitem[SS83]{SS83} \emph{Е. В. Щепин, М. А. Штанько.} Спектральный критерий вложимости компактов в евклидовы пространства, Труды Ленинградской Международной Топологической конференции. Л.: Наука, 1983. С.~135-142.}

\newcommand{\ssnt}{\bibitem[SS92]{SS92} \emph{J.~Segal and S.~Spie\.z.} Quasi embeddings and embeddings of polyhedra in $\R^m$,  Topol. Appl., 45 (1992) 275--282.}

\newcommand{\sszt}{\bibitem[SS03]{SS03} \emph{F. W. Simmons and F. E. Su.}
Consensus-halving via theorems of Borsuk-Ulam and Tucker, Math. Social Sciences 45 (2003) 15–25. \url{https://www.math.hmc.edu/~su/papers.dir/tucker.pdf}.}

\newcommand{\ssot}{\bibitem[SS13]{SS13} \emph{M. Schaefer and D. \v Stefankovi\v c.} Block additivity of $\Z_2$-embeddings. In Graph drawing, volume 8242 of Lecture Notes in Comput. Sci., 185--195.
Springer, Cham, 2013. \url{http://ovid.cs.depaul.edu/documents/genus.pdf}}

\newcommand{\sstt}{\bibitem[SS23]{SS23} \emph{A. Skopenkov and O. Styrt,} Embeddability of joinpowers, and minimal rank of partial matrices, arXiv:2305.06339.}

\newcommand{\sssne}{\bibitem[SSS]{SSS} \emph{J. Segal, A. Skopenkov and S. Spie\. z.}
Embeddings of polyhedra in $\R^m$ and the deleted product obstruction, Topol. Appl., 85 (1998), 225-234.}

\newcommand{\sstnf}{\bibitem[SST95]{SST95} \emph{R. S. Simon, S. Spie\. z and H. Toru\'nczyk.}
T\lowercase{HE EXISTENCE OF EQUILIBRIA IN CERTAIN GAMES, SEPARATION FOR FAMILIES OF CONVEX FUNCTIONS
AND A THEOREM OF BORSUK-ULAM TYPE}, Israel J. Math 92 (1995) 1--21.}

\newcommand{\sstzt}{\bibitem[SST02]{SST02} \emph{R. S. Simon, S. Spie\. z and H. Toru\'nczyk.}
E\lowercase{QUILIBRIUM EXISTENCE AND TOPOLOGY IN SOME REPEATED GAMES WITH INCOMPLETE INFORMATION},
Trans. Amer. Math. Soc. 354:12 (2002) 5005-5026.}

\newcommand{\stez}{\bibitem[ST80]{ST80} * {\it H.~Seifert and W.~Threlfall.}
A textbook of topology, v~89 of {\em Pure and Applied Mathematics}.
Academic Press, New York-London, 1980.}


\newcommand{\stzs}{\bibitem[ST07]{ST07} * \emph{А. Скопенков и А. Телишев.}
И вновь о критерии Куратовского планарности графов, Мат. Просвещение, 11 (2007), 159--160.}

\newcommand{\stzse}{\bibitem[ST07]{ST07} * \emph{A. Skopenkov and A. Telishev}, Once again on the Kuratowski graph planarity criterion, Mat. Prosveschenie, 11 (2007), 159-160. arXiv:0802.3820.}

\newcommand{\stos}{\bibitem[ST17]{ST17} \emph{A. Skopenkov  and M. Tancer,}
Hardness of almost embedding simplicial complexes in $\R^d$, Discr. Comp. Geom., 61:2 (2019), 452--463. arXiv:1703.06305.}

\newcommand{\stno}{\bibitem[ST91]{ST91} \emph{S.~Spie\. z and H.~Toru\'nczyk}, Moving compacta in $\R^m$ apart,
Topol. Appl. 41 (1991), 193--204.}

\newcommand{\sttt}{\bibitem[St24]{St24} \emph{M. Starkov,} An example of an `unlinked' set of $2k+3$ points in $2k$-space, arXiv:2402.09002.}

\newcommand{\sunt}{\bibitem[Su]{Su} * \emph{Д. Судзуки.} Основы дзэн-буддизма. Наука дзэн --- ум дзэн. Киев: Преса Украiни. 1992.}

\newcommand{\stwh}{\bibitem[SW]{SW} * \url{http://www.map.mpim-bonn.mpg.de/Stiefel-Whitney_characteristic_classes}}

\newcommand{\sz}{\bibitem[SZ05]{SZ} \emph{T. Sch\"oneborn and G. Ziegler}, The Topological Tverberg Theorem and Winding Numbers, J. Comb. Theory, Ser. A, 112:1 (2005) 82--104, arXiv:math/0409081.}

\newcommand{\szno}{\bibitem[Sz91]{Sz91} \emph{A.~Sz\"ucs,} On the cobordism groups of immersions and embeddings,
Math. Proc. Camb. Phil. Soc., 109 (1991) 343--349.}


\newcommand{\ta}{\bibitem[Ta]{Ta} * Handbook of Graph Drawing and Visualization. ed. by R. Tamassia, CRC Press, 2016.}


\newcommand{\tanfo}{\bibitem[Ta94]{Ta94} \emph{K. Taniyama,} Cobordism, homotopy and homology of graphs in $\R^3$,
Topology 33:3 (1994), 509--523.}

\newcommand{\tanf}{\bibitem[Ta95]{Ta95} \emph{K. Taniyama,} Homology classification of spatial embeddings of a graph, Topol. Appl. 65 (1995) 205--228.}

\newcommand{\tazz}{\bibitem[Ta00]{Ta00} \emph{K. Taniyama,} Higher dimensional links in a simplicial complex embedded in a sphere, Pacific Jour. of Math. 194:2 (2000), 465-467.}

\newcommand{\theo}{\bibitem[Th81]{Th81} * \emph{C.~Thomassen,} Kuratowski's theorem, J.~Graph. Theory 5 (1981), 225--242.}

\newcommand{\tooo}{\bibitem[To11]{To11} \emph{Tonkonog D.} Embedding 3-manifolds with boundary into closed 3-manifolds, Topol. Appl. 158 (2011), 1157-1162. arXiv:1003.3029.}


\newcommand{\tsbzf}{\bibitem[TSB]{TSB} \emph{D. M. Thilikos, M. Serna and H. L. Bodlaender},
Cutwidth I: A linear time fixed parameter algorithm, J. of Algorithms, 56:1 (2005), 1--24.}


\newcommand{\tsbzfd}{\bibitem[TSB05']{TSB05'} \emph{D. M. Thilikos, M. Serna and H. L. Bodlaender},
Cutwidth II: , J. of Algorithms, 56:1 (2005), 25--49.}



\newcommand{\umse}{\bibitem[Um78]{Um78} \emph{B. Ummel.} The product of nonplanar complexes does not imbed in 4-space, Trans. Amer. Math. Soc., 242 (1978) 319--328.}




\newcommand{\vant}{\bibitem[Va92]{Va92} * \emph{V.~A.~Vassiliev.} Complements of discriminants of smooth maps: Topology and applications, Amer. Math. Soc., Providence, RI, 1992 (рус. перевод: В. А. Васильев, Топология дополнений к дискриминантам, Фазис, Москва, 1997).}

\newcommand{\val}{\bibitem[Val]{Val} * \url{https://en.wikipedia.org/wiki/Valknut}}


\newcommand{\vi}{\bibitem[Vi]{Vi} * \emph{O. Viro.}
Some integral calculus based on Euler characteristic, Lect. Notes in Math. 1346.}

\newcommand{\vizt}{\bibitem[Vi02]{Vi02} * \emph{Э. Б. Винберг.} Курс алгебры. Москва. Факториал Пресс. 2002.}

\newcommand{\vizteng}{\bibitem[Vi02]{Vi02} * \emph{E. B. Vinberg.} A Course in Algebra. Graduate Studies in Mathematics, vol. 56. 2003.}

\newcommand{\vinhzs}{\bibitem[VINH07]{VINH07} * \emph{О. Я. Виро, О. А. Иванов, Н. Ю. Нецветаев и В. М. Харламов.}
Элементарная топология, МЦНМО. 2007.}

\newcommand{\vktt}{\bibitem[vK32]{vK32} \emph{E.~R.~van~Kampen}, Komplexe in euklidischen R\"aumen, Abh. Math. Sem. Hamburg, 9 (1933) 72--78; Berichtigung dazu, 152--153.}

\newcommand{\kafo}{\bibitem[vK41]{vK41} \emph{E. R. van Kampen,} Remark on the address of S. S. Cairns,
in Lectures in Topology, 311--313, University of Michigan Press, Ann Arbor, MI, 1941.}

\newcommand{\vo}{\bibitem[Vo96]{vo96} \emph{A. Yu. Volovikov,} On a topological generalization of the Tverberg theorem. Math. Notes 59:3 (1996), 324--326.}

\newcommand{\vopns}{\bibitem[Vo96v]{Vo96v} \emph{A. Yu. Volovikov,} On the van Kampen-Flores Theorem.
Math. Notes 59:5 (1996), 477--481.}


\newcommand{\vznt}{\bibitem[VZ93]{VZ93} \emph{A. Vu\v ci\'c and R. T. \v Zivaljevi\'c}, Note on a conjecture of Sierksma, Discr. Comput. Geom. 9 (1993), 339-349.}

\newcommand{\vzzn}{\bibitem[VZ09]{VZ09} \emph{S. T. Vre\'cica and R. T. \v Zivaljevi\'c},  Chessboard complexes
indomitable, J. of Comb. Theory, Ser. A 118:7 (2011), 2157--2166. arXiv:0911.3512.}


\newcommand{\walst}{\bibitem[Wa62]{Wa62} \emph{C.~T.~C.~Wall}, Classification of $(n-1)$-connected $2n$-manifolds, Ann. of Math., 75 (1962) 163--189.}


\newcommand{\wallss}{\bibitem[Wa67]{Wa67} \emph{C.~T.~C.~Wall.} Classification problems in differential topology, IV, Thickenings, Topology 1966. 5. P. 73--94.}

\newcommand{\waldss}{\bibitem[Wa67m]{Wa67m} \emph{F. Waldhausen.} Eine Klasse von 3-dimensional Mannigfaltigkeiten, I. Invent. Math. 1967. 3. P.~308-333.}

\newcommand{\walsz}{\bibitem[Wa70]{Wa70} \emph{C. T. C. Wall,} Surgery on compact manifolds,
1970, Academic Press, London.}

\newcommand{\wess}{\bibitem[We67]{We67} \emph{C.~Weber.} Plongements de poly\`edres dans le domain metastable, Comment. Math. Helv. 42 (1967), 1--27.}

\newcommand{\whit}{\bibitem[Wl]{Wl} * \url{https://en.wikipedia.org/wiki/Whitehead_link}}

\newcommand{\winum}{\bibitem[Wn]{Wn} * \url{https://en.wikipedia.org/wiki/Winding_number}}

\newcommand{\wrss}{\bibitem[Wr77]{Wr77} \emph{P. Wright.} Covering 2-dimensional polyhedra by 3-manifolds spines.
Topology. 16 (1977), 435--439.}

\newcommand{\wufe}{\bibitem[Wu58]{Wu58} \emph{W. T. Wu.} On the realization of complexes in a euclidean space (in Chinese): I, Sci Sinica, 7 (1958) 251--297; II, Sci Sinica, 7 (1958) 365--387; III, Sci Sinica, 8 (1959) 133--150.}

\newcommand{\wufn}{\bibitem[Wu59]{Wu59} \emph{W.~T.~Wu.} On the isotopy of a finite complex in Euclidean space, I, II, Science Record, N.S. 3:8 (1959) 342--347, 348--351.}

\newcommand{\wusf}{\bibitem[Wu65]{Wu65} * \emph{W. T. Wu.} A Theory of Embedding, Immersion and Isotopy of Polytopes in an Euclidean Space. Peking: Science Press, 1965.}


\newcommand{\yann}{\bibitem[Ya99]{Ya99} \emph{Z. Yang.} Computing Equilibria and Fixed Points: The Solution of Nonlinear Inequalities, Kluwer, Springer Science + Business Media, 1990.}


\newcommand{\zesz}{\bibitem[Ze60]{Ze60} \emph{E. C. Zeeman}, Unknotting spheres in five dimensions, Bull. Amer. Math. Soc. 66 (1960) 198.
\linebreak
\url{https://www.ams.org/journals/bull/1960-66-03/S0002-9904-1960-10431-4/S0002-9904-1960-10431-4.pdf}}

\newcommand{\z}{\bibitem[Ze]{Z} * \emph{E. C. Zeeman}, A Brief History of Topology, UC Berkeley, October 27, 1993, On the occasion of Moe Hirsch's 60th birthday, \url{http://zakuski.utsa.edu/~gokhman/ecz/hirsch60.pdf}.}

\newcommand{\zioz}{\bibitem[Zi10]{Zi10} * \emph{D. \v Zivaljevi\'c}, Borromean and Brunnian Rings,
\url{http://www.rade-zivaljevic.appspot.com/borromean.html}.}

\newcommand{\zioo}{\bibitem[Zi11]{Zi11} * \emph{G. M. Ziegler}, 3N Colored Points in a Plane, Notices of the Amer. Math. Soc., 58:4 (2011), 550-557.}


\newcommand{\zot}{\bibitem[Zi13]{Z13} \emph{A. Zimin.} Alternative proofs of the Conway-Gordon-Sachs Theorems, arXiv:1311.2882.}

\newcommand{\zss}{\bibitem[ZSS]{ZSS} * Элементы математики в задачах: через олимпиады и кружки к профессии.
Сборник под редакцией А. Заславского, А. Скопенкова и М. Скопенкова. М.: МЦНМО, 2018.
Обновляемая версия части книги: \url{http://www.mccme.ru/circles/oim/materials/sturm.pdf}.}

\newcommand{\zsse}{\bibitem[ZSS]{ZSS} Elements of mathematics via problems: from olympiades and math circles to a profession (in Russian), editors A. Zaslavsky, A. Skopenkov, and M. Skopenkov. MCCME, Moscow, 2018,
updated part of the book: \url{http://www.mccme.ru/circles/oim/sturm.pdf}.}


\newcommand{\zu}{\bibitem[Zu]{Zu} \emph{J. Zung.} A non-general-position Parity Lemma,
\url{http://www.turgor.ru/lktg/2013/1/parity.pdf}.}






\bibitem[AK20]{AK20}  \emph{B. H. An, B. Knudsen,}  On the second homology of planar graph braid groups, J. of Topology, 15:2 (2022) 666--691. arXiv:2008.10371. 

\bfzn

\bibitem[Bi21]{Bi21} {\it A. I. Bikeev,} Criteria for integer and modulo 2 embeddability of graphs to surfaces, arXiv:2012.12070v2.

\bibitem[Di08]{Di08} * \emph{T. tom Dieck,} Algebraic topology, EMS Textbooks in Mathematics, 
EMS, Z\"urich, 2008.

\bibitem[Dz25]{Dz25} \emph{E. Dzhenzher,} Symmetric 1-cycles in the deleted product of a graph, Topol. Appl. (2025) 109277.

\aronly{
\bibitem[DGN+]{DGN+} * \emph{S. Dzhenzher, T. Garaev, O. Nikitenko, A. Petukhov, A. Skopenkov, A. Voropaev,} Low rank matrix completion and realization of graphs: results and problems, arXiv:2501.13935.
}

\bibitem[DS22]{DS22}  \emph{S. Dzhenzher and A. Skopenkov,} A quadratic estimation for the K\"uhnel conjecture on embeddings, arXiv:2208.04188.



\bibitem[FH10]{FH10}  \emph{M. Farber, E. Hanbury}. Topology of Configuration Space of Two Particles on a Graph, II. Algebr. Geom. Topol. 10 (2010) 2203--2227. arXiv:1005.2300.

\bibitem[FK19]{FK19} \emph{R. Fulek, J. Kyn{\v{c}}l,} $\Z_2$-genus of graphs and minimum rank of partial symmetric matrices,
35th Intern. Symp. on Comp. Geom. (SoCG 2019), Article No. 39; pp. 39:1--39:16, \linebreak
\url{https://drops.dagstuhl.de/opus/volltexte/2019/10443/pdf/LIPIcs-SoCG-2019-39.pdf}.
We refer to numbering in arXiv version: arXiv:1903.08637.



\bibitem[HNT]{HNT} \emph{H. van der Holst, S. Norine, and R. Thomas}, On 2-cycles of graphs, arXiv:1711.04232. 


\bibitem[Ho07]{Ho07} \emph{H. van der Holst,} Algebraic characterizations of outerplanar and planar graphs, European J. Combin., 28:8 (2007), 2156--2166.

\bibitem[Ho09]{Ho09} \emph{H. van der Holst,} A polynomial-time algorithm to find a linkless embedding of a graph, J. Combin. Theory Ser. B, 99:2 (2009), 512--530.


 

\bibitem[KP11]{KP11}  \emph{Ki Hyoung Ko, Hyo Won Park,} Characteristics of graph braid groups,  Discrete Comput. Geom. 48 (2012), 915--963. arXiv:1101.2648. 

\bibitem[KS20]{KS20} \emph{R. Karasev and A. Skopenkov.} Some `converses' to intrinsic linking theorems, Discr. Comp. Geom., 
70:3 (2023), 921--930, arXiv:2008.02523.

\bibitem[KS21]{KS21} * \emph{E. Kogan and A. Skopenkov.} A short exposition of the Patak-Tancer theorem on non-embeddability 
of $k$-complexes in $2k$-manifolds,  arXiv:2106.14010.

\bibitem[KS21e]{KS21e} \emph{E. Kogan and A. Skopenkov.} Embeddings of $k$-complexes in $2k$-manifolds and minimum rank of partial symmetric matrices, arXiv:2112.06636v2.

\bibitem[Ku23]{Ku23} \emph{W. K\"uhnel.} Generalized Heawood Numbers, The Electronic Journal of Combinatorics, 30:4 (2023) \#P4.17.
 
\bibitem[Ky16]{Ky16} \emph{J. Kyn{\v{c}}l.} Simple realizability of complete abstract topological graphs simplified, Discrete Comput. Geom. 64 (2020) 1--27. 
arXiv:1608.05867. 


\bibitem[Ma03]{Ma03} * \emph{J.~Matou{\v{s}}ek.} Using the {B}orsuk-{U}lam theorem:
Lectures on topological methods in combinatorics and geometry. Springer Verlag, 2008.

\bibitem[Me06]{Me06} \emph{S. A. Melikhov}, The van Kampen obstruction and its relatives, 	Proc. Steklov Inst. Math 266 (2009), 142-176 (= Trudy MIAN 266 (2009), 149-183), arXiv:math/0612082.

\aronly{\bibitem[MPS]{MPS} \emph{M. J. Pelsmajer, M. Schaefer, and D. Stefankovic},  Removing even crossings, J. Combin. Theory Ser. B, 97:4 (2007), 489--500.}

\bibitem[Ni00]{Ni00} \emph{R. Nikkuni.} The second skew-symmetric cohomology group and spatial embeddings of graphs, J. Knot Theory Ram. 9 (2000), 387--411.


\bibitem[PT19]{PT19} \emph{P. Pat\'ak and M. Tancer.} Embeddings of $k$-complexes into $2k$-manifolds. Discrete Comput. Geom. 71 (2024), 960--991. arXiv:1904.02404.

\aronly{
\pstz
}

\bibitem[RS72]{RS72} * \emph{C. P. Rourke and B. J. Sanderson,}
\newblock Introduction to Piecewise-Linear Topology,
\newblock \emph{Ergebn.\ der Math.} 69, Springer-Verlag, Berlin, 1972.

\aronly{\bibitem[Sa91]{Sa91} \emph{K. S. Sarkaria.} A one-dimensional Whitney trick and Kuratowski's graph planarity criterion, Israel J.~Math. 73 (1991), 79--89.
}

\bibitem[Sc13]{Sc13} * \emph{M. Schaefer.} Hanani-Tutte and related results. In Geometry --- intuitive, discrete, and convex, Bolyai Soc. Math. Stud., 24 (2013), 259--299.
\url{http://ovid.cs.depaul.edu/documents/htsurvey.pdf} 

\bibitem[Sk18]{Sk18} * \emph{A. Skopenkov.} Invariants of graph drawings in the plane. Arnold Math. J., 6 (2020) 21--55; full version: arXiv:1805.10237.

\bibitem[Sk20]{Sk20} * \emph{A. Skopenkov,} Algebraic Topology From Geometric Standpoint (in Russian), MCCME, Moscow, 2020 (2nd edition).
Update of a part: \url{http://www.mccme.ru/circles/oim/obstruct.pdf} .
Part of the English translation: \url{https://www.mccme.ru/circles/oim/obstructeng.pdf}.


\bibitem[Sk24]{Sk24} \emph{A. Skopenkov.} Embeddings of $k$-complexes in $2k$-manifolds and minimum rank of partial symmetric matrices, arXiv:2112.06636v4.

\bibitem[Sk]{Sk} * \emph{A. Skopenkov.} Algebraic Topology From Algorithmic Standpoint, draft of a book, mostly in Russian,
\url{http://www.mccme.ru/circles/oim/algor.pdf}.

\jonly{\bibitem[SS23]{SS23} \emph{A. Skopenkov and O. Styrt,} Embeddability of joinpowers and minimal rank of partial matrices, arXiv:2305.06339v4.}

\aronly{\bibitem[SSv3]{SSv3} \emph{A. Skopenkov and O. Styrt,} Embeddability of joinpowers and minimal rank of partial matrices, arXiv:2305.06339v3.}

\bibitem[ST03]{ST03} \emph{R. Shinjo, K. Taniyama,} Homology classification of spatial graphs by linking numbers and Simon invariants, Topol. Appl. 134 (2003), 53--67.


\bibitem[vK32]{vK32} \emph{E.~R.~van~Kampen}, Komplexe in euklidischen R\"aumen, Abh. Math. Sem. Hamburg, 9 (1933) 72--78; Berichtigung dazu, 152--153.

\end{thebibliography}
\end{document}